\documentclass{article}
\usepackage{amssymb,latexsym,amsthm,amsmath,pifont,graphicx,anysize,multicol,slashed}
\usepackage{amscd}
\usepackage{hyperref}
\marginsize{1.25in}{1.25in}{0.75in}{1.25in}

\newtheorem{thm}{Theorem}[section]

\newtheorem{lemma}[thm]{Lemma}

\newtheorem{define}{Definition}
\newtheorem{prop}[thm]{Proposition}
\newtheorem{cor}[thm]{Corollary}%
\theoremstyle{definition}
\newtheorem{remark}[thm]{Remark}
\newtheorem{ex}[thm]{Example}

\DeclareMathOperator{\ind}{ind} \DeclareMathOperator{\indx}{Ind}
\DeclareMathOperator{\Ind}{Ind} \DeclareMathOperator{\ch}{ch}
\DeclareMathOperator{\sch}{sch} \DeclareMathOperator{\Sch}{Sch}
\DeclareMathOperator{\rk}{rk} \DeclareMathOperator{\coker}{coker}
\DeclareMathOperator{\Diff}{Diff}
\DeclareMathOperator{\F}{\mathbf{\mathcal F}}

\newcommand{\eqn}{\begin{equation}}
\newcommand{\eeqn}{\end{equation}}
\newcommand{\R}{\mathbb{R}}
\newcommand{\C}{\mathbb{C}}
\newcommand{\Z}{\mathbb{Z}}
\newcommand{\Q}{\mathbb{Q}}
\newcommand{\g}{\mathfrak{g}}
\newcommand{\dd}{\slashed {\partial}}
\newcommand{\DD}{\mathcal{\slashed{D}}}
\newcommand{\M}{\mathcal{M}}
\newcommand{\MM}{\widetilde{\mathcal{M}}}
\newcommand{\J}{\mathcal{J}}

\begin{document}
\title{The Index Bundle for a Family of Dirac-Ramond Operators}
\author{Chris Harris\footnote{\texttt{charris@math.miami.edu}}}
\date{}
\maketitle

\begin{abstract}
\noindent We study the index bundle of the Dirac-Ramond operator
associated with a family $\pi: Z \to X$ of compact spin manifolds.
We view this operator as the formal twisted Dirac operator $\dd
\otimes \bigotimes_{n=1}^{\infty}S_{q^n}TM_{\C}$ so that its index
bundle is an element of $K(X)[[q]]$. When $p_1 (Z) = 0$, we derive
some explicit formulas for the Chern character of this index
bundle using its modular properties. We also use the modularity to
identify our index bundle with an $L(E_8)$ bundle in a special
case.
\end{abstract}

\everymath{\displaystyle}

\section{Introduction}

In the 1980's several genera valued in the ring of modular forms were introduced.  The elliptic genera originated in work of Ochanine \cite{Och}
and Landweber and Stong \cite{LandStong}, and were soon after given an interpretation through physics.  By extending the path integral proof of the Atiyah-Singer
index theorem to a certain supersymmetric nonlinear sigma model, it was shown by Alvarez, Killingback, Mangano, and Windey in \cite{Alvarez89} and \cite{Alvarez87}
that the elliptic genera could be viewed as the equivariant index of a certain twisted Dirac like operator on the free loop space.
They also showed that the index of the untwisted version of this operator, known as the Dirac-Ramond operator, could be computed; this produced another genus called
the string genus.  The string genus is also known as the Witten genus because independently around the same time Witten
greatly further elucidated the relationships between quantum field theory, genera, and index theory in \cite{Witten87} and \cite{Witten88}.
More recently, Alvarez and Windey have shown that
their earlier work can be extended to the case of families of Dirac-Ramond operators.  The index theorem proved in
\cite{Alvarez2010} for the Dirac-Ramond operator is the elliptic analogue of the original cohomological version of the Atiyah-Singer
index theorem for a family of Dirac operators \cite{ASIV}.

No one has given a general mathematical construction for the Dirac-Ramond operator
on the full loop space,
though there are some partial results (see e.g. \cite{Spera} for the case when the manifold is flat).   Each manifold can be embedded in its loop
space via constant maps, and the Dirac-Ramond operator can be defined
rigorously on the normal bundle given by this embedding as in \cite{Segal} or \cite{Taubes}.  It is well known (see \cite{HBJ}, for instance) that the index of this operator is given by
a certain formal sum of twisted $\widehat A$-genera.  From this perspective, the
index of the Dirac-Ramond operator can also be obtained by considering the operator as a formal sum of twisted Dirac operators, or equivalently as the usual Dirac operator twisted by a formal
sum of bundles.  This is the viewpoint we will take below in the family case, defining the index bundle of the Dirac-Ramond operator to be the formal sum of index bundles
from the appropriate twisted Dirac operators.  Note that such an object has been considered by Liu and Ma in \cite{LiuMa} and subsequent work where they achieved considerable rigidity results.

The outline of this paper is as follows.  We start in Section $2$ by describing the setup and making precise the Dirac-Ramond operator as a formal sum of operators.  We
will then show that the cohomological family index theorem from our formal sum agrees with that of the index theorem shown by Alvarez and Windey.  In Section $3$, we will further investigate
the family index.  In the case of certain families of string manifolds, we then use a method, different from that used in \cite{Alvarez2010}, to see that the Chern character
of the family index is given by cohomology classes with coefficients
in the ring of (quasi)modular forms.  We will show by way of an example how one can use modularity to generate relations between various index bundles associated to the various
operators used in defining the Dirac-Ramond operator.  These sorts of relations are similar to the ``anomaly cancellation formulas" which arise in physics.  Some results
of this same type, but on the level of differential forms, were derived using elliptic genera in \cite{hanliu}.  We also make use of the theory of Jacobi-like forms in order to
derive an explicit formula describing the Chern character of the index bundle for the Dirac-Ramond operator in terms of the components
of the Chern character for some twisted Dirac operators and Eisenstein series.  In Section $4$, we will apply the above formalism in the case where the manifold has dimension $8$ and
the parameterizing space has dimension less than $16$.  We can then use the formulas from Section $3$ to show that under certain conditions
the index of the family of Dirac-Ramond operators is equivalent in $(K(X)\otimes \Q) [[q]]$ to a vector bundle associated with the basic representation of the loop group for $E_8$.

\section{The Index Theorem}
\subsection{General Setup}
In the following, by manifold we will always mean a smooth connected manifold without boundary.
Let $M$ be a compact spin manifold of even dimension $d$.  For any vector bundle $W\rightarrow M$ we have a sequence of vector bundles $\{ W_n\}$ defined by the generating series
\begin{equation}\label{genseries}
\bigotimes_{j=1}^{\infty}S_{q^j}W_{\C}=\sum_{n=0}^{\infty}q^nW_n.
\end{equation}
Here $W_{\C}$ denotes the complexification of $W$ and $S_t(W_{\C})=\C + tW_{\C}+t^2S^2(W_{\C})+...$ is a formal power series with vector bundle coefficients.  One has, for instance,
\begin{align*}
W_0&=\C \\
W_1&=W_{\C} \\
W_2&=S^2W_{\C}\oplus W_{\C} \\
W_3&=S^3W_{\C}\oplus (W_{\C}\otimes W_{\C})\oplus W_{\C}\\
&\vdots
\end{align*}

We can use these bundles as coefficients for the usual chiral Dirac operator $\dd$ on $M$ to obtain twisted Dirac operators $\dd^{W_n} : C^{\infty} (S^+\otimes W_n)\rightarrow C^{\infty} (S^-\otimes W_n)$
where $S^+$ and $S^-$ are the spinor bundles of positive and negative chirality, respectively.  It is a classical result that an elliptic operator $P$ on a compact manifold
is Fredholm.  The index of $P$ is defined as
\begin{equation}\label{indexdef}
\Ind P := \dim \ker P - \dim \coker P.
\end{equation}
The twisted Dirac operators $\dd^{W_n}$ are known to be elliptic.  The Atiyah-Singer index theorem \cite{AS} provides the formula for the index of these operators as
\begin{equation}
\Ind \dd^{W_n}=\int_M \widehat A (TM) \ch (W_n).
\end{equation}
Let $x_1,...,x_{d/2}$ denote the roots for the tangent bundle $TM$.  The $x_i$'s are defined so that the total Pontryagin class $p(TM)=1+p_1(TM)+p_2(TM)+...$
satisfies $$p(TM)=\prod_{i=1}^{d/2}(1+x_i^2).$$
Then $\widehat A(p_1,p_2,...)\in H^*(M,\Q)$ is determined by the expression \begin{scriptsize}$\prod_{i=1}^{d/2}$\end{scriptsize}\begin{small}$\frac{x_i/2}{\sinh (x_i/2)}$\end{small}.

Now we define the Dirac-Ramond operator on $M$ as the formal series of operators $$\DD=\sum_{n=0}^{\infty}q^n\dd^{TM_n}.$$  Equivalently, if we allow for ``$q$-vector bundles," that is, formal power
series in $q$ whose coefficients are vector bundles, then we can view it as the the twisted Dirac operator,
$$\DD : C^{\infty} (S^+\otimes \bigotimes_{n=1}^{\infty}S_{q^n}TM_{\C})\rightarrow C^{\infty} (S^-\otimes \bigotimes_{n=1}^{\infty}S_{q^n}TM_{\C}).$$
Extending the definition of index to each coefficient of the formal operator $\DD$ as
\begin{equation}\label{inddr}
\Ind \DD=\sum_{n=0}^{\infty}q^n\Ind \dd^{TM_n}
\end{equation}
the calculation of this index is a well known calculation
using formulas involving characteristic classes and elliptic functions.  We will essentially do it below when we prove the family index theorem.  The result can be written
in any of the following ways:

\begin{align}\label{inddr2}
 \Ind \DD&= \frac{q^{d/24}}{\eta (q)^d}\int_M \prod_{i=1}^{d/2} x_i\frac{\theta '(0,q)}{\theta (x_i,q)}  \\
&= \frac{q^{d/24}}{\eta (q)^d}\int_M \prod_{i=1}^{d/2} \frac{x_i}{\sigma (x_i,q)}e^{G_2(q)p_1(M)}\notag\\
&=\frac{q^{d/24}}{\eta (q)^d}\int_M \prod_{i=1}^{d/2}\exp\left(\sum_{n=1}^{\infty}\frac{2}{2n!}G_{2n}(q)x_i^{2n}\right )\notag
\end{align}
where $\theta(x,q)$ is the Jacobi theta function, $\sigma (x,q)$ is the Weierstrass sigma function, and $G_{2n}(q)$ is the Eisenstein series of weight $2n$.
The definitions of all of these functions and some of the relations between them
can be found in the appendix.  With the assumption that $M$ is spin the equation (\ref{inddr}) implies that (\ref{inddr2}) is an element of $\Z [[q]]$.
The expression on the RHS of (\ref{inddr2}) makes sense for any compact oriented manifold and more generally lies in $\Q[[q]]$.
The factor of
\begin{equation}\label{afactor}
\frac{q^{d/24}}{\eta (q)^d}=\prod_{j=1}^{\infty}(1-q^j)^{-d}
\end{equation}
is not very interesting.  We will often choose to omit it,
defining $$\ind \DD := \frac{\eta(q)^d}{q^{d/24}} \Ind \DD.$$
With this normalization, $\ind \DD$ is the $q$ expansion of a modular form of weight $d/2$ when $p_1(M)=0$ and a quasimodular form of the same weight otherwise.
The second line of (\ref{inddr2}) indicates that $\ind \DD$ is an elliptic analogue of the index
of the usual Dirac operator: $\sinh (x/2)$ is a periodic function on $i\R$ with simple zeros at every point in the lattice $2\pi i \Z$, while the Weierstrass
sigma function $\sigma (x,\tau)$ is a (quasiperiodic) extension of this to the complex plane and the two dimensional lattice $2\pi i (\Z \oplus \tau \Z)$.
Here we have used the change of variables $q=e^{2\pi i \tau}$, common in the study of modular forms, where $\tau$ is an element of the upper half complex plane.
Based on this similarity we will define the following for a rank $r$ real vector bundle $W$ over $M$ with roots $w_1,...,w_{r/2}$
$$\widehat a (W, q)=\prod_{i=1}^{r/2}\frac{w_i}{\sigma (w_i,q)}\in H^*(M,\Q[[q]])=H^*(M,\Q)[[q]].$$
In the case when $M$ is a string manifold (that is, $M$ is a manifold with $p_1(M)=0$)\footnote{Since we are working only with rational cohomology, we ignore the subtler condition that $p_1(M)/2$ should equal $0$ in integral cohomology.}, then we can write
\begin{equation}
\ind \DD = \int_M \widehat a (TM,q).
\end{equation}

\subsection{Family Index}

The index of a single Fredholm operator is an integer.  The index of a family $P$ of elliptic operators $P_x:C^{\infty} (E_x)\rightarrow C^{\infty}(F_x)$ parameterized by a compact space $X$
is an element of $K(X)$ defined as follows.  For each $x\in X$, one has the finite dimensional subspaces $\ker P_x$ and $\coker P_x$ of $C^{\infty}(E_x)$ and $C^{\infty} (F_x)$, respectively.
If as $x$ varies
these vector spaces gave rise to the vector bundles $\ker P$ and $\coker P$ over $X$, the desired generalization of (\ref{indexdef}) would be
\begin{equation}\label{indexdeffam}
\indx P := \ker P - \coker P \in K(X).
\end{equation}
This need not be the case, however, as there is the possibility that, as $x$ varies, the dimensions of $\ker P_x$ and $\coker P_x$ may jump.  It is only
the difference in their dimensions that is fixed by invariance of the index under continuous perturbations.  Thus $\ker P$
and $\coker P$ need not be vector bundles.  Regardless, there is a way to define the index in the case of a family (see section 2 of \cite{ASIV})
which reduces to (\ref{indexdeffam}) in the case when $\ker P_x$ and $\coker P_x$ are of constant dimension.

By a family of compact spin manifolds we will mean a triple $\mathbf{\mathcal F} =(\pi, Z, X)$ such that $\pi : Z \rightarrow X$ is a fiber bundle
whose fibers $Y_x:=\pi ^{-1}(x)$ are compact spin
manifolds all diffeomorphic, and with spin structure isomorphic, to some compact spin manifold $Y$.  We will assume that $X$ is a compact spin \footnote{The spin condition on $X$ is really only necessary in the last section.}
manifold, so that $Z$ is compact and spin as well.
We will always assume all manifolds are even dimensional and denote the dimension of $Y$ by $m$.  We will also denote the vertical tangent bundle
(the bundle which is tangent to the fibers, i.e. $\text{ker }\pi_*$) by $V\rightarrow Z$.  Applying the generating
sequence (\ref{genseries}) to $V$ we get a sequence of vector bundles $V_n$ over $Z$.

Given a family of compact spin manifolds, a family of Dirac operators can be constructed by taking the ordinary Dirac operator on each fiber $Y_x$.  Moreover,
the vector bundles $V_n$ over $Z$ restrict to the fiber $Y_x$ and give rise to families of twisted Dirac operators $\dd_x ^{V_n} : \Gamma (S_x^+\otimes V_n|_{Y_x})\rightarrow \Gamma (S_x^-\otimes V_n|_{Y_x})$.
 The index of the family $\dd_x^{V_n}$ is denoted $\indx \dd ^{V_n}$ and is an element of $K(X)$.  The original
formula for the Chern character of the index bundle for a general
family of elliptic operators is given in \cite[Theorem 5.1]{ASIV}.
For the case of a family of twisted Dirac
operators the formula is (see \cite[Theorem 4.17]{B} or
\cite[Corollary 15.5]{LM})
\begin{equation}\label{ordfamthm}
\ch (\indx \dd^{V_n})=\int_Y \widehat A
(V) \ch (V_n) \in H^*(X,\Q)
\end{equation}
where $\int_Y : H^*(Z,\Q)\rightarrow H^{*-m}(X,\Q)$ is the integration over the fibers map.

Given a family of compact spin manifolds, we will define a family of Dirac-Ramond operators to be the $q$-series with the coefficient of $q^n$
being the family of Dirac operators twisted by $V_n$
$$\DD:=\sum_{n=0}^{\infty}q^n\dd^{V_n}.$$
We are thus led to our main object of study
\begin{define}
Let $\F=(\pi,Z, X)$ be a family of compact spin manifold.  Let $V_n$ be the sequence of vector bundles
obtained as above.  The index of the family of Dirac-Ramond operators is defined to be
\begin{equation}\label{indxdefn}
 \indx\DD := \sum_{n=0}^{\infty}q^n \indx \dd ^{V_n} \in K(X)[[q]].
\end{equation}
\end{define}
The Chern character $\ch : K(X)\rightarrow H^*(X,\Q)$ extends naturally to a map $K(X)[[q]] \rightarrow H^*(X,\Q)[[q]]$.
In particular,
$$\ch (\indx\DD)=\sum_{n=0}^{\infty}q^n \ch(\indx \dd ^{V_n})\in H^*(X,\Q)[[q]].$$
We can calculate this Chern character easily by applying the usual Atiyah-Singer index theorem for families (\ref{ordfamthm}) coefficient by coefficient.
\begin{prop}
Let $\F=(\pi,Z, X)$ be a family of compact spin manifold whose fibers $Y_x:=\pi^{-1}(x)$ are of even dimension $m$.  Let $V\rightarrow Z$ denote
the vertical bundle.Then
\begin{equation}\label{DRindthm}
\ch(\indx \DD)  = \frac{q^{m/24}}{\eta(q)^m}\int_Y \hat a (V,q)e^{G_2(q)p_1(V)}.
\end{equation}
\end{prop}
\begin{proof}
The equation (\ref{DRindthm}) is an equation in $H^*(X, \Q)[[q]]$.  We establish
(\ref{DRindthm}) by showing equality for each coefficient of $q^n$.  On the LHS, we get the coefficient of $q^n$ by applying the Chern character to $\indx\dd^{V_n}$
which is \begin{scriptsize}$\int_Y$\end{scriptsize}$ \widehat A (V) \ch (V_n)$, see (\ref{ordfamthm}).  Let $y_1,...,y_{m/2}$ be the formal Chern variables for the vertical bundle $V$.  Making use
of the formulas from the appendix we have
\begin{align*}
\frac{q^{m/24}}{\eta(q)^m}\hat{a} (V,q)e^{G_2(q)p_1(V)} &=\prod_{i=1}^{m/2}\frac{q^{1/12}}{\eta(q)^2}\frac{y_i}{\sigma (y_i,q)}e^{G_2(q)y_i^2}=\prod_{i=1}^{m/2}\frac{q^{1/12}}{\eta(q)^2}y_i\frac{\theta ' (0,q)}{\theta (y_i,q)}\\
&=\prod_{i=1}^{m/2}\frac{y_i/2}{\sinh (y_i/2)}\prod_{j=1}^{\infty}\frac{1}{(1-q^je^{y_i})(1-q^je^{-y_i})} \\
&=\hat A (V) \ch\left(\bigotimes_{j=1}^{\infty}S_{q^j}(V_{\C})\right).
\end{align*}
After integrating over the fibers we can identify the coefficient of $q^n$ as\begin{scriptsize} $\int_Y$\end{scriptsize}$ \widehat A (V) \ch (V_n)$ and the result follows.
\end{proof}
We are interested in the formal definition (\ref{indxdefn}) because our formula (\ref{DRindthm})
matches that of \cite{Alvarez2010}.  In the spirit of that paper, from now on, given a family of Dirac-Ramond operators we will denote the
Chern character of the index bundle by
\begin{equation}\label{schdefine}
\Sch(\F; q):=\ch\left(\sum_{n=0}^{\infty}q^n \indx \dd ^{V_n}\right)=\sum_{n=0}^{\infty}q^n \ch(\indx \dd ^{V_n}) \in H^*(X,\Q)[[q]],
\end{equation}
using the letter ``S" because this Chern character is of a ``stringy" version of the usual index bundle.  By the proposition we have
\begin{equation}
\Sch(\F; q)  = \frac{q^{m/24}}{\eta(q)^m}\int_Y \hat a (V,q)e^{G_2(q)p_1(V)}.
\end{equation}
At times we will prefer to use instead
\begin{equation}
\sch(\F; q)  := \frac{\eta(q)^m}{q^{m/24}}\Sch (q)=\int_Y \hat a (V,q)e^{G_2(q)p_1(V)}.
\end{equation}

\section{Modular Properties}
From now on we will restrict ourselves to so called string families of compact spin manifolds $\F=(\pi, Z, X)$, that is, families of compact spin manifolds with $p_1(Z)=0$.
We will denote by $\M^k$ (respectively $\MM ^k$) the vector space of weight $k$, level $1$ modular (respectively quasimodular) forms having rational $q$ coefficients and by $\M =\oplus_{k=1}^{\infty}\M^k$
(respectively $\MM =\oplus_{k=1}^{\infty}\MM ^k$) the full ring of modular (respectively quasimodular) forms.  It is well known that $\M \simeq \Q [E_4, E_6]$ and $\MM \simeq \M [E_2]$, where
$E_n$ is the normalized Eisenstein series of weight $n$.

Viewing $\M$ and $\MM$ as $\Q$-vector spaces we have $H^*(X,\M)\simeq H^*(X,\Q)\otimes \M$ and similarly $H^*(X,\MM)\simeq H^*(X,\Q)\otimes \MM$.
Both $H^*(X,\M)$ and $H^*(X,\MM)$ are naturally identified with subrings of $H^*(X,\Q[[q]])\simeq H^*(X,\Q)[[q]]$.
The integration over the fibers map above was extended to $H^*(Z,\Q)[[q]]\rightarrow H^{*-m}(X,\Q)[[q]]$
coefficient by coefficient.  It is not hard to see that this restricts to the map
\begin{align*}
\int_Y: H^*(Z,\MM^*)&\rightarrow H^{*-m}(X,\MM^*)\\
\omega \otimes f(q)&\mapsto\left(\int_Y\omega\right)\otimes f(q).
\end{align*}
Using this, we will first show in this section that the image of $\sch(\F ; q)$ actually lies in
the subring $H^*(X,\M)$, whenever $X$ is a string manifold, meaning $p_1(X)=0$.  If $X$ is not string then the image will be in $H^*(X,\MM)$.
As a consequence of this we will be able to derive many relationships between the homogeneous
components of $\ch (\indx \dd ^{V_n})$ for various $n$.  Finally, we will obtain a formula for the case when $X$ is not string which expresses $\Sch(\F; q)$ entirely in terms
of modular forms, the components of $\ch (\indx \dd ^{V_n})$ for various $n$, and $p_1(X)$.

\subsection{The Case When The Parameterizing Space is String}\label{hasdetinit}
Given a string family of manifolds $\pi:Z \rightarrow X$, it follows that each of the manifolds $Y_x=\pi ^{-1}(x)$ is a string manifold.  To see this,
let $i:Y_x\hookrightarrow Z$ be the inclusion map.  Then the restriction $i^*TZ$ splits as $TY_x\oplus N_x$, where $N_x\rightarrow Y$ is the normal
bundle to $Y_x$ in $TZ$.  Since $\pi:Z \rightarrow X$ is locally trivial, $N_x$ is trivial.  Therefore, $p_1(N_x)=0$ and hence
$$p_1(Y_x)=i^*p_1(TZ)=0.$$
One can always split the tangent bundle to $Z$, though non-canonically, as $TZ=V\oplus\pi ^*(TX)$.  Consequently,
$p_1(V)=-\pi^*p_1(X)$ and thus
\begin{equation}\label{p1n}
\int_Y p_1(V)^n=(-1)^n\int_Y \pi^*p_1(X)^n=(-1)^np_1(X)^n\int_Y 1 = 0 \quad \text{for all }n\in \mathbb N.
\end{equation}
In \cite{Alvarez2010}, the assumption $p_1(X)=0$ was present throughout and the following theorem and its corollary were already noticed, though the modularity
of the graded components was shown by other means.
\begin{thm}\label{thm31}
Let $\F=(\pi, Z, X)$ be a string family of compact spin manifolds with fibers $Y_x=\pi ^{-1}(x)$ of dimension $m$.  If we expand
$\sch(\F; q)$ into its homogeneous components as an element of $H^*(X,\Q[[q]])$,
$$\sch(\F; q)=\sch_0(\F;q)+\sch_1(\F;q)+...$$
then $\sch_j(\F;q)\in H^{2j}(X,\MM^{\frac{m}{2}+j})$, and if $p_1(X)=0$ then $\sch_j(\F;q)\in H^{2j}(X,\M^{\frac{m}{2}+j})$.  Moreover,
$$\sch_0(\F;q)=\ind \DD_Y,$$
i.e. $\sch_0(\F;q)$ is the string, or Witten, genus of $Y$.

\end{thm}
\begin{proof}
Since $p_1(Z)=0$
\begin{equation}
\sch(\F; q)=e^{-G_2(q)p_1(X)}\int_Y\hat a (V,q).
\end{equation}
Recall the change of variable $q=e^{2\pi i \tau}$ noted above.  We will write functions, such as the Weierstrass sigma function $\sigma$,
interchangeably as $\sigma(z,q)$ or $\sigma(z, \tau)$, thinking of them as a formal power series in the former case and a function on the upper half
plane in the latter.  The function $\widehat a (z,\tau) =\frac{z}{\sigma (z,\tau)}=\exp$\begin{footnotesize}$ \left(\sum_{n=2}^{\infty}\frac{2}{2n!}G_{2n}(\tau)z^{2n}\right)$\end{footnotesize}
is even, has $\widehat a(0,\tau)=1$, and is homolomorphic at $z=0$.  Thus it has a Taylor series expansion
\begin{equation}\label{littlet}
\widehat a (z,\tau)=1+f_2(\tau)z^2+f_4(\tau)z^4+...
\end{equation}
We have
$$\widehat a \left(\frac{z}{c\tau + d},\frac{a\tau+b}{c\tau+d}\right)=\exp \left(\sum_{n=2}^{\infty}\frac{2}{2n!}G_{2n}\left(\frac{a\tau+b}{c\tau+d}\right)\left(\frac{z}{c\tau +d}\right)^{2n}\right)=\widehat a (z,\tau).$$
Thus
$$\widehat a \left(\frac{z}{c\tau + d},\frac{a\tau+b}{c\tau+d}\right)=1+f_2\left(\frac{a\tau+b}{c\tau+d}\right)\left(\frac{z}{c\tau + d}\right)^2+f_4\left(\frac{a\tau+b}{c\tau+d}\right)\left(\frac{z}{c\tau + d}\right)^4+...$$
is equal to (\ref{littlet}).  Equating coefficients of $z^{2n}$ it follows that $f_{2n}\left(\frac{a\tau+b}{c\tau+d}\right)=(c\tau +d)^{2n}f_{2n}(\tau)$.
Hence (\ref{littlet})
can be considered as a formal power series with coefficients in $\M^*$ and moreover
the coefficient of $z^{2n}$ lies in $\M^{2n}$.
Letting $y_1,...,y_{m/2}$ denote the roots for the vertical bundle $V$, it follows that
\begin{align*}
\prod_{n=1}^{m/2}\widehat a (y_n,\tau)&=\prod_{n=1}^{m/2}(1+f_2(\tau)y_n^2+f_4(\tau)y_n^4+...)\\
&=1+f_2(\tau)p_1(V)+f_4(\tau)p_1(V)^2+(f_2(\tau)^2-2f_4(\tau))p_2(V)+...
\end{align*}
That is, if we let $\prod_{n=1}^{m/2}\widehat a (y_n,\tau)=1+a_2(\tau)+a_4(\tau)+...$ be the decomposition into homogeneous components in cohomology, then
$a_{2n}(\tau)\in H^{4n}(Z,\M^{2n})$.  Now,
\begin{equation}
\sch_j(\F;\tau)=\int_Y a_{\frac{m}{2}+j}(\tau)
\end{equation}
and integration over the fibers defines a map $H^{m+2j}(Z,\M^{\frac{m}{2}+j})\rightarrow H^{2j}(X,\M^{\frac{m}{2}+j})$.  Thus the $p_1(X)=0$
part of the theorem follows.  If $p_1(X)\neq 0$ then the factor $e^{-G_2(q)p_1(X)}$
will produce cohomology classes with coefficients that are polynomials in $G_2(q)$ with coefficients in $\M$, which is precisely $\widetilde \M$.  Note that since $G_2(q)$ has modular weight $2$ and $p_1(X)$ has
cohomological degree $4$, a homogeneous cohomology class of degree $2j$ in $\sch(\F ; q)$ will still have quasimodular weight $m/2+j$.  The last part follows by evaluating on a point in $X$.
\end{proof}
The (quasi)modularity forces many relations between the characteristic classes of the index bundles at each level.  When the dimension $m$ of the manifold $Y$ and the degree in cohomology are small
we have the following result.

\begin{cor}\label{coronmod}
Assume the setup of the previous theorem with $p_1(X)=0$.  Let $j$ be such that $\frac{m}{2}+j\leq 14$ and $\frac{m}{2}+j\neq 12$.
Then, for each $n$ there is $c(j,n)\in \Q$ such that
$$\ch_{j}(\indx \dd ^{V_n})=c(j,n)\ch_{j}(\indx \dd) \quad \in H^{2j}(X,\Q).$$
\end{cor}
\begin{proof}
From the previous theorem we know that the degree $2j$ component $\sch_j(\F;q)$ of $\sch(\F; q)$ is an element of $H^{2j}(X, \M^{\frac{m}{2}+j})=H^{2j}(X,\Q)\otimes\M^{\frac{m}{2}+j}$.
For $m/2+j=4,6,8,10,14$, have $\M ^{m/2+j}=\Q E_{\frac{m}{2}+j}(q)$.  Thus
\begin{equation}\label{2kcoh}
\sch_{j}(\F;q)=\omega E_{\frac{m}{2}+j}(q), \ \ \text{for some } \omega \in H^{2j}(X,\Q).
\end{equation}
Now also,
\begin{equation}\label{2kcoh2}
\sch_{j}(\F;q) = \frac{\eta (q)^m}{q^{m/24}} \Sch_{j}(\F;q) =  \prod_{n=1}^{\infty}(1-q^n)^m \sum_{i=0}^{\infty}\ch_{j} (\indx \dd ^{V_i})q^i
\end{equation}
Comparing the $q^0$ term in (\ref{2kcoh}) and (\ref{2kcoh2}) we see $\omega =\ch_{j} (\indx \dd)$.  Hence
\begin{align}\label{schexpand}
\sum_{i=0}^{\infty}\ch_{j} (\indx \dd ^{V_i})q^i&=\Sch_{j}(\F;q) \\
                                               &=\frac{q^{m/24}}{\eta (q)^m}\sch_{j}(q)= \ch_{j} (\indx \dd ) q^{m/24}\frac{E_{\frac{m}{2}+j}(q)}{\eta (q)^m}.\notag
\end{align}
The proportionality factor $c(j,n)$ are then extracted from $E_{\frac{m}{2}+j}(q)$\begin{scriptsize}$\prod_{n=1}^{\infty}$\end{scriptsize}$(1-q^n)^{-m}$ by taking the coefficient of $q^n$.
Since $\M ^{\frac{m}{2}+j}=\{0\}$ for all other values of $m/2+j$ up to $14$ except $12$, the result follows (trivially) for these values.
\end{proof}
The above does not apply in the case of $m/2+j=12$ since $\M ^{12}$ is $2$ dimensional.  However, the methods of the proof above can be extended.  Essentially,
one sees that if the dimension of $\M ^{m/2+j}$ is $s > 1$, then for all $n\geq s$, $\ch _{j} (\indx \dd ^{V_n})$ will be a linear combination of
$\ch _{j} (\indx \dd ^{V_0})$,...,$\ch _{j} (\indx \dd ^{V_{s-1}})$.

We will illustrate all of this very explicitly when dim $Y=8$ for low degrees in cohomology.  Assume also that $p_1(X)=0$ and that dim $X$ is divisible by $4$.
In this situation every time a cohomology class of degree $2k$ appears, it will be multiplied by a
modular form of weight $4 +k$.  From Corollary \ref{coronmod} we know
\begin{align}\label{expn111}
\sch_{\leq 6}(\F ;q)=\ch_0(\indx \dd) E_4(q)&+\ch_2(\indx \dd) E_6(q)+\ch_4(\indx \dd) E_8(q)\\
                                                                     &+\ch_{6}(\indx \dd) E_{10}(q).\notag
\end{align}
The degree $0$ component of the Chern character is just the (virtual) rank of the index bundle.  Set $\nu_i:=\rk \indx \dd^{V_i}$.  In particular $\nu_0$ is the virtual rank of the
index bundle of the untwisted Dirac operator, i.e. the index of the usual Dirac operator on $Y$.  Using (\ref{schexpand}) for $k=0$ one obtains
$$\nu_0q^{1/3}\frac{E_4(q)}{\eta (q)^8}=\sum_{i=0}^{\infty}\nu_iq^i.$$
We can use
\begin{equation}\label{deg0ind}
\nu_0\frac{E_4(q)}{\eta (q)^8}=\nu_0 j(q)^{1/3}=\nu_0q^{-1/3}(1 + 248q + 4124q^2+...).
\end{equation}
Then comparing the right hand sides of the previous two equations we get relations like
\begin{align}\label{ch0s}
\rk \indx \dd^{V_{\C}}&= 248 \rk \indx \dd \\
\rk \indx \dd^{S^2V_{\C}\oplus V_{\C}}&= 4124 \rk \indx \dd\notag\\
&\vdots\notag
\end{align}
Putting $k=2$ in (\ref{schexpand}) gives
$$\ch_2(\indx \dd)q^{1/3}\frac{E_6(q)}{\eta (q)^8}=\sum_{i=0}^{\infty}\ch_2(\indx \dd^{V_i})q^i.$$
Noting also that $q^{1/3}\frac{E_6(q)}{\eta (q)^8}=(1-496q-20620q^2+...)$ gives
\begin{align}\label{ch4s}
\ch_2( \indx \dd^{V_{\C}})&= -496 \ch_2( \indx \dd) \\
\ch_2( \indx \dd^{S^2V_{\C}\oplus V_{\C}})&= -20620\ch_2(\indx \dd)\notag\\
&\vdots\notag
\end{align}
There are similar relations between the cohomology classes in degree $8$ and $12$ that one could write out.  The degree $16$ cohomology classes, however,  are not all proportional.  This is because they have coefficients in $\M ^{12}$, which
is now $2$ dimensional.
We use the basis $\{E_4(q)^3-728\Delta (q),\Delta (q)\}$ of $\M^{12}$.  We have
$$\Sch_{8}(\F ;q)=\alpha q^{1/3}\frac{E_4(q)^3-728\Delta (q)}{\eta(q)^8} + \beta q^{1/3}\frac{\Delta(q)}{\eta(q)^8}=\alpha (1 + 196732q^2+...) + \beta (q-16q^2+...).$$
which is to be compared with
$$\sum_{i=0}^{\infty}\ch_{8} (\indx \dd ^{V_i})q^i=\ch_{8} (\indx \dd ^{V_0})+\ch_{8} (\indx \dd ^{V_1})q+...$$
and the values of $\alpha$ and $\beta$ are easily read off so that (\ref{expn111}) can be extended by including
\begin{align*}
\sch_{8}(\F ;q)&=\ch_{8}(\indx \dd) (E_4(q)^3-728\Delta (q))+\ch_{8}(\indx \dd^{V_{\C}})\Delta (q).
\end{align*}
We could expand this in powers of $q$.  Rather than the cohomology classes $\ch_8(\indx \dd^{V_n})$ being proportional for all $n\in \mathbb N$, as before, we would see that
$\ch_8(\indx \dd^{V_n})$ for $n\geq 2$ is a linear combinations
of $\ch_8(\indx \dd)$ and $\ch_8(\indx \dd^{V_1})$.

Moving to degree $20$ in cohomology will give weight $14$ modular forms.  Here again the Corollary \ref{coronmod} can be applied and the next term is simply in $\sch(\F ; q)$ is
$$\sch_{10}(\F ;q)=\ch_{10}(\indx \dd) E_{14}(q).$$

The method for all higher degrees is a straightforward generalization of the degree $16$ case.
Let $r$ denote the dimension of $\mathcal M ^k$, then there is a basis $\{ f_0(q), f_1(q)\Delta(q), ...,f_{r-1}(q)\Delta(q)^{r-1}\} $
for $\mathcal M ^k$ with $f_i(0)=1$ and, of course, $\Delta(q)^i=q^i+...$.
Simple linear algebra can be used to find in $\mathcal M ^k$ a basis $\{ \phi_0(q),...,\phi_{r-1}(q)\}$ instead which satisfies $q^{1/3}\frac{\phi_i(q)}{\eta(q)^8}=q^i+\mathcal O (q^{r})$.
Working in this basis it is easy to see $$\Sch_{k-4}(\F ;q)=\frac{q^{1/3}}{\eta(q)^8}\left(\ch_{k-4} (\indx \dd ^{V_0})\phi_0(q)+...+\ch_{k-4} (\indx \dd ^{V_{r-1}})\phi_{r-1}(q)\right).$$

We restate the last part of the example in more general terms.
\begin{prop}\label{genprop}
Let $Z \rightarrow X$ be a string family of compact spin manifolds where each $Y_x=\pi^{-1}(X)$ has even dimension $m$ and $p_1(X)=0$.
Let $s_j=\dim \mathcal M ^{\frac{m}{2}+j}$, and $\{ \phi_0(q),...,\phi_{s_j-1}(q)\}$ be the basis for $\M^{\frac{m}{2}+j}$ which satisfies $q^{m/24}\frac{\phi_i(q)}{\eta(q)^m}=q^i+\mathcal O (q^{s_j})$.
Then $$\Sch_{j}(\F ;q)=\frac{q^{m/24}}{\eta(q)^m}\left(\ch_{j} (\indx \dd ^{V_0})\phi_0(q)+...+\ch_{j} (\indx \dd ^{V_{s_j-1}})\phi_{s_j-1}(q)\right).$$
\end{prop}


Before moving on we wish to point out some connection with the preceding and anomaly cancellation.  Thinking of the index bundle for the Dirac operator
as the formal difference $\indx \dd=\ker \dd - \text{coker} \dd$ in $K(X)$, one can define the determinant line bundle
$$\det \dd = \det (\ker \dd)\otimes \det (\coker \dd)^* \in K(X).$$
As in defining the index bundle, this is not strictly true as the dimension of each
space $\ker \dd$ and $\text{coker } \dd$ may individually jump.  However the determinant line bundle $\det \dd \rightarrow X$ can still be defined (see \cite{FreedDet}) and one has
\begin{equation}
c_1(\det \dd)=c_1(\indx \dd)=\ch_1(\indx \dd)\in H^2(X,\Q).
\end{equation}
In physics, this characteristic class is referred to as an anomaly.  The Proposition \ref{genprop} produces many ``anomaly cancellation formulas."
Examples of these formulas are
\begin{equation}\label{anom}
c_1(\det \dd^{V_n})=\alpha (n)c_1(\det \dd), \ \text{ for some } \alpha(n)\in \Q.
\end{equation}
which follow directly from Corollary \ref{coronmod} whenever $\dim Y \leq 24$, except when $\dim Y=20$.  The equation (\ref{anom}) holds nontrivially
when $\dim Y=6,10,14,18,$ or $22$.  The operator $\dd^{V_1}=\dd^{V_{\C}}:C^{\infty}(S^+\otimes V_{\C})\rightarrow C^{\infty}(S^-\otimes V_{\C})$ is almost
what is known as
the Rarita-Schwinger operator.  If $\dim Y=6$ and $p_1(Z)=p_1(X)=0$ \footnote{It is sufficient to require just that $p_1(V)=0$ instead.}
the following formula holds.
$$c_1(\det \dd ^{V_{\C}})=246c_1(\det \dd)$$
If $\dim Y=20$ or $\dim Y >24$ one needs to appeal more directly to Proposition \ref{genprop}.  For instance, in the case that $\dim Y=20$ the proposition gives
$$\Sch_1(\F ;q)=\frac{q^{5/6}}{\eta (q)^{20}}\left(\ch_1(\indx \dd)(E_4(q)^3-740\Delta (q))+\ch_1(\indx^{V_1})\Delta(q)\right)$$
from which it follows that
$$c_1(\det \dd^{S^2V_{\C}\oplus V_{\C}} )=196870 c_1(\det \dd ) - 4 c_1(\det \dd^{V_{\C}} ).$$

\subsection{Some Computational Motivation}\label{compmotiv}
Now we will drop the assumption that $p_1(X)= 0$ (but maintain $p_1(Z)=0$). The results will now be quasimodular rather than modular.  The dimensions of
the space of quasimodular forms grow much quicker as one goes to higher weights.  Because of this one might expect much less
rigidity in the structure of the index bundle for a family of Dirac-Ramond operators.  However, we will see that this is not the case.  In the next section
we will state and prove a theorem which generalizes ($\ref{genprop}$) in the case $p_1(X) \neq 0 $.  The formula within the theorem is very complicated and in this
section we will demonstrate the formula in some special cases. In the case where the fiber $Y$ has dimension $8$ some direct computation (see the appendix)
shows the following

\begin{align}\label{start}
& \sch_{\leq 6}(\F ;q)=\notag \\
& \nu_0\left(E_4(q)+\frac{1}{(4)_1}E_4'(q)\left(\frac{p_1(X)}{2}\right)+\frac{1}{2! (4)_2}E_4''(q)\left(\frac{p_1(X)}{2}\right)^2+\frac{1}{3! (4)_3}E_4'''(q)\left(\frac{p_1(X)}{2}\right)^3\right)\notag\\
+&\ch_2(\indx \dd) \left(E_6(q)+\frac{1}{(6)_1}E_6'(q)\left(\frac{p_1(X)}{2}\right)+\frac{1}{2! (6)_2}E_6''(q)\left(\frac{p_1(X)}{2}\right)^2\right)\\
+&\ch_4(\indx \dd) \left(E_8(q)+\frac{1}{(8)_1}E_8'(q)\left(\frac{p_1(X)}{2}\right)\right)\notag\\
+&\ch_{6}(\indx \dd) E_{10}(q)\notag
\end{align}
where $(k)_n=k(k+1)...(k+n-1)$ is the Pochhammer symbol and $f'(q)= q$\begin{small}$\frac{d}{dq}$\end{small}$f(q)$.  Note that all of the terms with a derivative are of order $q$,
so when $q\rightarrow 0$ one obtains $\Sch_{\leq 6}(\F ;0)=\nu_0+\ch_2(\indx \dd)+\ch_4(\indx \dd)+\ch_6(\indx \dd)$, as expected from ($\ref{schdefine}$).
Of course, the $p_1(X)\rightarrow 0$ limit reduces to the previous
case ($\ref{expn111}$).

Using (\ref{start}), $\Sch(\F ; q)=\sch(\F ; q)$\begin{scriptsize}$\prod_{n=1}^{\infty}$\end{scriptsize}$(1-q^n)^{-8}$, and the $q$ expansion of the Eisenstein series the following relations
are obtained
\begin{align}\label{p1not0}
\ch_2(\indx \dd^{V_{\C}})&=-496\ch_2(\indx \dd)+30\nu_0p_1(X) \notag\\
\ch_2(\indx \dd^{S^2V_{\C}\oplus V_{\C}})&=-20620\ch_2(\indx \dd)+780\nu_0p_1(X)\\
\ch_4(\indx \dd^{V_{\C}})&=488\ch_4(\indx \dd_z)-42p_1(X)\ch_2(\indx \dd)+\frac{3}{2}\nu_0p_1(X)^2\notag\\
\ch_4(\indx \dd^{S^2V_{\C}\oplus V_{\C}})&=65804\ch_4(\indx \dd)-3108p_1(X)\ch_2(\indx \dd)+66\nu_0p_1(X)^2.\notag
\end{align}

The next relevant degree is $16$, where the modular forms become weight $12$ and now have an extra dimension;
we expect something interesting to happen.  The result can be written

\begin{align}\label{sch8deg16}
\sch_{8}(\F ;q) &=\ch_{8}(\indx \dd) (E_{4}(q)^3-728\Delta (q)  )+\ch_{8}(\indx \dd^{V_{\C}}) \Delta(q) \notag \\
&+\frac{\nu _0}{4!\cdot (4)_4}\left(E_4^{(4)}(q)-240\Delta (q)\right)\left(\frac{p_1(X)}{2}\right)^4  \\
&+\frac{\ch_2(\indx \dd) }{3!\cdot (6)_3}\left(E_6^{(3)}(q)+504\Delta (q)\right)\left(\frac{p_1(X)}{2}\right)^3 \notag \\
&+\frac{\ch_4(\indx \dd) }{2!\cdot (8)_2 }\left(E_8^{(2)}(q)-480\Delta (q)\right)\left(\frac{p_1(X)}{2}\right)^2+\frac{\ch_{6}(\indx \dd) }{(10)_1}\left(E_{10}'(q)+264\Delta (q)\right)\left(\frac{p_1(X)}{2}\right).\notag
\end{align}
What we see is that the pattern of coefficients in (\ref{start}) continues, but there are extra terms.  Notice that each of the terms $E_4^{(4)}(q)-240\Delta (q)$,...,$E_{10}'(q)+264\Delta (q)$ are all order $q^2$.
This makes sense since after multiplying by \begin{scriptsize}$\prod_{n=1}^{\infty}$\end{scriptsize}$(1-q^n)^{-8}$ the second term will give $\ch_{8}(\indx \dd^{V_{\C}}) q$ and this is entirely what the coefficient of $q$ in $\Sch_{8}(q)$ should be.

When $\dim Y=6$, one obtains similar results.  The same calculations as above show that in this case
\begin{align}\label{dim6fiber}
&\sch_{\leq 7}(\F ;q)=\notag \\
& \ch_1(\indx \dd)\left(E_4(q)+\frac{1}{(4)_1}E_4'(q)\left(\frac{p_1(X)}{2}\right)+\frac{1}{2! (4)_2}E_4''(q)\left(\frac{p_1(X)}{2}\right)^2+\frac{1}{3! (4)_3}E_4'''(q)\left(\frac{p_1(X)}{2}\right)^3\right)\notag\\
+&\ch_3(\indx \dd) \left(E_6(q)+\frac{1}{(6)_1}E_6'(q)\left(\frac{p_1(X)}{2}\right)+\frac{1}{2! (6)_2}E_6''(q)\left(\frac{p_1(X)}{2}\right)^2\right)\\
+&\ch_5(\indx \dd) \left(E_8(q)+\frac{1}{(8)_1}E_8'(q)\left(\frac{p_1(X)}{2}\right)\right)\notag\\
+&\ch_{7}(\indx \dd) E_{10}(q)\notag
\end{align}

\subsection{A General Formula}\label{pfsect}

Let $\mathfrak h$ denote the complex upper half plane and $Hol(\mathfrak h)$ denote the space of holomorphic functions on $\mathfrak h$.
We will make extensive use of the following
\begin{define}
A Jacobi-like form of weight $k$ and index $\lambda$ is an element of $F(z,\tau)\in Hol (\mathfrak h)[[z]]$ such that
\begin{equation}\label{jaclike}
F\left (\frac{z}{c\tau + d}, \frac{a\tau + b}{c\tau + d}\right)= (c\tau +d)^k \exp\left(\frac{c\lambda}{c\tau +d}\frac{z^2}{2\pi i}\right)F(z,\tau)
\end{equation}
for all $\left( \begin{array}{cc}
a & b \\
c & d \end{array} \right) \in SL(2,\Z)$.  We will denote the collection of Jacobi-like forms of weight $k$ and index $\lambda$ by $\mathcal J _{k,\lambda}$.
\end{define}
\noindent Jacobi-like forms satisfy one of the two transformations properties which essentially characterize Jacobi forms.  The foundations
for Jacobi forms were laid out in \cite{EZ} and the generalization of Jacobi forms
to Jacobi-like forms was introduced in \cite{CMZ} and \cite{Zagier94}.
\begin{ex}\label{jfexample}
Given a modular form $f\in \M ^k$ one can verify $F(z,\tau)=z^{n}f(\tau)e^{-G_2(\tau)\lambda z^2}$ is a Jacobi-like form of weight $k-n$ and index $\lambda/2$.
\end{ex}
As with modular forms, Jacobi-like forms can be defined on subgroups of the modular group, as well, but that will not be necessary for us here.
Since Jacobi-like forms are
invariant under the transformation $\tau \mapsto \tau +1$ they have a Fourier expansion in terms of $e^{2\pi i \tau}$.  Setting $q=e^{2\pi i \tau}$ we will sometimes
write a Jacobi-like form $F(z,\tau)$ instead as $F(z,q)$.
Given a Jacobi-like form $F(z,\tau)=$\begin{scriptsize}$\sum_{j=0}^{\infty}$\end{scriptsize}$\chi_{j}(\tau)z^{j}\in \J_{k,\lambda}$
we show in the appendix that $\chi_{j}$ is a modular form of weight $k+j$ if the index is zero and a quasimodular form of the same weight otherwise.
To deal with the slight technicality that
our definition of $\M$ and $\MM$ are only for modular forms with rational $q$ expansion, we will
denote by $\J_{\Q,k,\lambda}$ the elements of $\J_{k,\lambda}$ with $\chi_{j}\in \MM ^{k+j}$.  Since $\MM^{k+j}$ is trivial when $k+j$ is odd,
the $\chi_j(\tau)$ are necessarily zero for half of the $j$'s.  Thus $\J_{k,\lambda}$ is the direct sum of the two subspaces
\begin{align*}
\J^+_{k,\lambda}:=\{F(z,\tau)\in \J_{k,\lambda}|F(z,\tau)&=\sum_{j=0}^{\infty}\chi_{2j}(\tau)z^{2j}\}\\
\J^-_{k,\lambda}:=\{F(z,\tau)\in \J_{k,\lambda}|F(z,\tau)&=\sum_{j=0}^{\infty}\chi_{2j+1}(\tau)z^{2j+1}\}.
\end{align*}
It is easy to verify that the following map is an isomorphism.
\begin{align}\label{plusminusisom}
\J^+_{k,\lambda}&\rightarrow \J^-_{k-1,\lambda}\\
F(z,\tau)&\mapsto zF(z,\tau).\notag
\end{align}
Given an $F(z,\tau)\in \mathcal J _{k,\lambda}$, evaluating (\ref{jaclike}) at $z=0$ shows that $F(0,\tau)$ is a modular form of weight $k$.  Thus there is a map
$\J^+ _{\Q,k,\lambda} \rightarrow \M^k$.
\begin{define}\label{CKlift}
Given $f\in \mathcal{M} ^k$ the Cohen-Kuznetsov series (or lift) of $f$ with index $\lambda$ is given by
\begin{equation}
\widetilde{f}(z, \tau) = \sum_{n=0}^{\infty}\frac{\lambda^nf^{(n)}(\tau)}{n!(k)_n}z^{2n}\in Hol(\mathfrak h)[[z]]
\end{equation}
where $(k)_n=(k+n-1)!/(k-1)!=k(k+1)...(k+n-1)$ is the Pochhammer symbol and $f^{(n)}(\tau):=\left(\frac{1}{2\pi i}\frac{d}{d\tau}\right)^nf(\tau)$.  If no mention is made of the index, it will be assumed that $\lambda=1$.
\end{define}
The derivation $D:=$\begin{small}$\frac{1}{2\pi i}\frac{d}{d\tau}$\end{small} corresponds under the change of variables $q=e^{2\pi i \tau}$ to the operator \begin{small}$q\frac{d}{dq}$\end{small}, which was used in (\ref{start}).
The ring $\MM$ has the nice property of being closed under differentiation.  In fact,
$D: \MM^*\rightarrow \MM^{*+2}$.  We illustrate a bit with the following

\begin{ex}
\begin{align}\label{E4CK}
&\widetilde E_4(z,q)=E_4(q)+\frac{1}{4}E_4'(q)z^2+\frac{1}{2!\cdot 4 \cdot 5}E_4''(q)z^4+\frac{1}{3!\cdot 4 \cdot 5\cdot 6}E_4'''(q)z^6+...\\
&=E_4(q)+\frac{1}{12}\left(E_4(q)E_2(q)-E_6(q)\right)z^2+\frac{1}{288}\left(E_4(q)^2-2E_6(q)E_2(q)+E_4(q)E_2(q)^2\right)z^4+...\notag
\end{align}
\end{ex}
It is shown in Section 3 of \cite{EZ} that if $f\in \M^k$ and $\widetilde f (z,\tau)$ is its Cohen-Kuznetsov lift with index $\lambda$, then
$\widetilde f (z,\tau)$ satisfies (\ref{jaclike}) and so $\widetilde f (z,\tau)\in \J^+ _{\Q,k,\lambda}$.  This justifies the term \emph{lift},
as the map $f \mapsto \widetilde f $ provides a section for the map $\J^+ _{\Q,k,m}\rightarrow \M^k$.
In fact, all elements of $\J_{\Q,k,\lambda}$ can be constructed from Cohen-Kuznetsov lifts via the following.
\begin{thm}\label{jlfseqs}
Given a Jacobi-like form $F(z,\tau)\in \J_{\Q,k,\lambda}$ there is a corresponding sequence of modular forms $\xi_0,\xi_1,\xi_2,...$ such that $\xi_{n}\in \M^{k+n}$ and
\begin{equation}\label{jlfck}
F(z,\tau)=\widetilde{\xi_0}(z, \tau)+z\widetilde{\xi_1}(z, \tau)+z^2\widetilde{\xi_2}(z, \tau)+...=\sum_{n=0}^{\infty}z^{n}\widetilde{\xi_{n}}(z, \tau)
\end{equation}
and $\widetilde{\xi_{n}}(z, \tau)$ is, as above, the Cohen-Kuznetsov lift of $\xi_{n}$ with index $\lambda$.  Conversely, for any sequence of modular forms $\xi_0,\xi_1,\xi_2,...$
satisfying $\xi_{n}\in \M^{k+n}$, the equation (\ref{jlfck}) defines an $F(z,\tau)\in\J_{\Q,k,\lambda}$.
\end{thm}
\begin{remark}\label{rmk}
\noindent Again, since $\M^{k+j}=\{0\}$ whenever $k+j$ is odd, half of the $\xi_j$'s necessarily vanish.  Note that in terms of the usual expansion
\begin{scriptsize}$\sum_{j=0}^{\infty}$\end{scriptsize}$\chi_{j}(q)z^{j}$ of an $F(z,\tau)\in \J_{\Q,k,\lambda}$,
the sequence of modular forms are given by
\begin{equation}\label{f2ns}
\xi_{n}(\tau)=\sum_{0\leq j \leq n/2}\frac{(-\lambda)^j(k+n-j-2)!}{j!(k+n-2)!}\chi^{(j)}_{n-2j}(\tau).
\end{equation}
See Section 3 of \cite{EZ} for the proof.
\end{remark}

Noticing the similarity of the coefficients in (\ref{start}) and (\ref{E4CK})
we are led to formulate a general theorem which puts $\sch(\F ; q)$ as something akin to a ``Jacobi-like form in \begin{footnotesize}$\frac{p_1(X)}{2}$\end{footnotesize}".  To make this more precise we first define a map
\begin{align*}
\mathcal J^+_{\Q,k,\lambda} &\rightarrow H^*(X,\MM)\\
F(z,q)=\sum_{j=0}^{\infty}\chi_{2j}(q)z^{2j}&\mapsto \sum_{j=0}^{\infty}\chi_{2j}(q)\left(\frac{p_1(X)}{2}\right)^{j}
\end{align*}
The map is well defined since $X$ is finite dimensional.  To maybe abuse notation, we will denote the image of an element $F(z,q)$ under this map by $F\left ( \frac{1}{2}p_1(X), q\right)$
and call it a cohomological Jacobi-like form, or CJLF for short.
Using this, we see that (\ref{start}) can be restated as
\begin{align}\label{startb}
\sch_{\leq 6}(\F ;q) = & \nu_0\widetilde {E_4} \left(\frac{p_1(X)}{2},q\right)_{\leq 12} +\ch_2(\indx \dd)\widetilde {E_6} \left(\frac{p_1(X)}{2},q\right)_{\leq 8}\\
&+\ch_4(\indx \dd) \widetilde {E_8} \left(\frac{p_1(X)}{2},q\right)_{\leq 4} +\ch_{6}(\indx \dd)\widetilde{E_{10}}\left(\frac{p_1(X)}{2},q\right)_{\leq 0}\notag
\end{align}
where for any CJLF $\textstyle F\left ( \frac{1}{2}p_1(X), q\right)$ we denote its projection onto degree at most $n$ in cohomology by $\textstyle F\left ( \frac{1}{2}p_1(X), q\right)_{\leq n}$.
Compare (\ref{startb}) to (\ref{expn111}).  The general formula for $\sch(\F ; q)$ given in the theorem
below can be better understood by reexamining what we have so far and what we get by moving to the next (nonzero) degree in cohomology.  The terms that one
gets from proposition \ref{genprop} in the $p_1(X)=0$ case, i.e. those in (\ref{expn111}), have modular form coefficients.  Taking the Cohen-Kuznetsov lift of these
modular forms and promoting them to CJLF's gives (\ref{startb}).  Moving to degree $16$ in cohomology we see from (\ref{sch8deg16}) that we get two more modular terms guaranteed by
Proposition \ref{genprop} and four other terms having modular (cusp) form coefficients which cancel the coefficient of $q$ in all terms containing $p_1(X)$ in (\ref{startb})
(making them order $q^2$).  The pattern then repeats each time you move to the next relevant degree in cohomology, there will be some new terms which arise
from Proposition \ref{genprop}, say $s$ of them, and another term which cancels the first $s$ coefficient of all other terms containing a $p_1(X)$.

To put things more formally we first need the following.
\begin{lemma}\label{sharp}
Let $\phi\in \M^k$ and put $s_{2j}=\dim \M^{k+2j}$.  Then there is a unique element
$$\phi^{\natural}(z,q)=\sum_{j=0}^{\infty}\chi_{2j}(q)z^{2j}\in \J^+_{\Q,k,1}$$
such that $\chi_0(q)=\phi(q)$ and for $j>0$
$$\chi_{2j}(q) \in q^{s_{2j}}\Q [[q]]\cap\MM^{k+2j}.$$
\end{lemma}
\begin{proof}
As in Theorem \ref{jlfseqs}, a sequence of modular forms $\{f_{2\ell}\}_{\ell\geq 0}$ such that $f_{2\ell}\in \M^{k+2\ell}$ uniquely defines an element $\phi^{\natural}(z,q)\in \J^+_{\Q,k,\lambda}$ via
$$\phi^{\natural}(z,q)=\widetilde{f_0}(z, q)+z^2\widetilde{f_2}(z, q)+z^4\widetilde{f_4}(z, q)+...$$
We set $f_0(q)=\phi(q)$ and define $f_{2\ell}(q)$ for $\ell>0$ recursively by requiring that
$$f_{2\ell}(q)+\sum_{n=0}^{j-1}\frac{f_{2n}^{(\ell -n)}(q)}{(\ell-n)!(k+2n)_{\ell -n}}\in q^{s_{2\ell}}\Q[[q]].$$  This
recursive equation has a unique solution for each $j$ since the equation puts $s_{2\ell}$ independent conditions on the $f_{2\ell}$ and $\M^{k+2\ell}$ is $s_{2\ell}$ dimensional.
\end{proof}
\noindent One can see that $0^{\natural}(z,q)=0$; a less trivial example follows.
\begin{ex}
$$E_4^{\natural}(z,q)=\widetilde {E_4}(z,q)-z^8\frac{240}{4! (4)_4}\widetilde {\Delta} (z,q)-z^{12}\left(\frac{240}{6!(4)_6} -\frac{240}{4! (4)_4}\frac{1}{2!(12)_2}\right)\widetilde{\Delta E_4}(z,q)+...$$
\end{ex}
Comparing the example with (\ref{startb}) and (\ref{sch8deg16}) we see that, up to degree $16$, the coefficient of $\nu_0$ in $\sch(\F ; q)$ is the CJLF $\textstyle E_4^{\natural}\left ( \frac{1}{2}p_1(X), q\right)$.
The following theorem asserts that this extends to all degrees in cohomology and that all the other components of the Chern characters of the index bundles of the
various twisted Dirac operators that show up in $\sch (q)$ also have coefficients that are a CJLF $\textstyle f^{\natural}\left ( \frac{1}{2}p_1(X), q\right)$ for some
modular form $f$.

\begin{thm}\label{bigthm}
Let $\F=(\pi, Z, X)$ be a string family of compact spin manifolds where each $Y_x=\pi^{-1}(x)$ has even dimension $m$.
Let  $s_j=\dim \M^{\frac{m}{2}+j}$.  Then
\begin{equation}\label{maineqn}
\Sch(\F ; q)=\frac{q^{m/24}}{\eta(q)^m}\sum_{j=0}^{\infty}\ch_{j}(\indx \dd^{V_0})\phi^{\natural}_{j,0}\left(\frac{p_1(X)}{2},q\right)+...+\ch_{j}(\indx \dd^{V_{s_j-1}})\phi^{\natural}_{j,s_j-1}\left(\frac{p_1(X)}{2},q\right)
\end{equation}
where for each $j$ the collection $\phi_{j,0}(q),...,\phi_{j,s_j-1}(q)$ is given by Proposition \ref{genprop}.
\end{thm}

\begin{proof}
Consider the polynomial rings $S=\Q[p_1,...,p_{m/2}]$ in the indeterminates $p_i$ of weight $4i$.  We regard $S$ as a subspace of the polynomial ring $\Q[y_1,...,y_{m/2}]$,
with indeterminates $y_i$ of weight $2i$, via the degree preserving injection $p_i\mapsto \sigma_i(y_1^2,...,y_{m/2}^2)$, where $\sigma_i$ is the $i$th elementary symmetric function.
Using the series expansion of the exponential function one obtains
$$\psi(z,q)=\prod_{i=1}^{m/2}\frac{zy_i}{\sigma(zy_i,q)}=\prod_{i=1}^{m/2}\exp\left(\sum_{n=2}^{\infty}\frac{2}{2n!}G_{2n}(q)(zy_i)^{2n}\right )$$
as an element of $(\Q[[q]]\otimes S)[[z]]$.
From the proof of Theorem \ref{thm31}, we see that $\psi(z,q)$ actually lies in a smaller space.  Namely, let $S=S^0\oplus S^4\oplus...$ be the decomposition
into homogeneous subspaces and set $\textstyle \mathcal R=\bigoplus_{j=0}^{\infty}\M^{2j}\otimes S^{4j}$; then $\psi(z,q)\in \mathcal R[[z]]$ and
the coefficient of $z^n$ is in $\M^{2n}\otimes S^{4n}$.

Let $r\in\{0,2\}$ be the reduction of $m$ modulo $4$.  Expanding
$$\psi(z,q)=1+z^2a_2(q;p_1)+z^4a_4(q;p_1,p_2)+...$$
we define
\begin{align}\label{psijl}
\Psi (z,q)&=\frac{1}{z^{m/2}}\left[\psi (z,q)-\left(1+z^2a_2(q;p_1)+...+z^{(m+r)/2-2}a_{(m+r)/2-2}(q;p_1,...)\right)\right]e^{G_2(q)p_1z^2}\notag\\
&=e^{G_2(q)p_1z^2}\sum_{j=(m+r)/4}^{\infty}a_{2j}(q;p_1,...)z^{2j-m/2}
\end{align}
Notice that the coefficient of $z^n$ in $\textstyle \sum_{j=(m+r)/4}^{\infty}a_{2j}(q;p_1,...)z^{2j-m/2}$ is an element of $\M^{\frac{m}{2}+n}\otimes S^{m+2n}$.
Set $\textstyle \widetilde{\mathcal R}=\bigoplus_{j=0}^{\infty}\MM^{2j}\otimes S^{4j}$; then we have $\Psi(z,q)\in \widetilde{\mathcal R}[[z]]$ and if
$$\Psi (z,q)=\chi _0(q;p_1,...)+\chi_1(q;p_1,...)z+\chi_2(q;p_1,...)z^2+...$$
then $\chi_{n}(q;p_1,...)\in \MM^{\frac{m}{2}+n}\otimes S^{m+2n}$.

Using Example \ref{jfexample}, we see from (\ref{psijl}) that for any $(y_1,...,y_{m/2})\in C^{m/2}$ it follows that
$\Psi$ is in either $\J^+_{m/2, -\frac{p_1}{2}}$ or $\J^-_{m/2,-\frac{p_1}{2}}$ depending on whether $r=0$ or $r=2$, respectively.
We can then form the modular combinations as in (\ref{f2ns})
\begin{equation}
\xi_{n}(q;p_1,...)=\sum_{0\leq j \leq n/2}\frac{(\frac{p_1}{2})^j(\frac{m}{2}+n-j-2)!}{j!(\frac{m}{2}+n-2)!}\chi^{(j)}_{n-2j}(q;p_1,...)\in\M^{m/2+n}\otimes S^{m+2n}
\end{equation}
and, as above, $\chi ^{(j)}=\left(q\frac{d}{dq}\right)^j\chi$.  Then as in (\ref{jlfck}) we have
$$\Psi(z,q)=\widetilde {\xi_0}(z,q)+z\widetilde {\xi}_1(z,q)+z^2\widetilde {\xi}_2(z,q)+...$$
where $$\widetilde {\xi}_n(z,q)=\sum_{\nu=0}^{\infty}\frac{(-\frac{p_1}{2})^{\nu}\xi_{n}^{(\nu)}(q;p_1,...)}{\nu!(m/2+n)_{\nu}}z^{2n}\in \widetilde{\mathcal R}[[z]]$$
is the Cohen-Kuznetsov lift of $\widetilde {\xi}_n$ with index \begin{footnotesize}$-\frac{p_1}{2}$\end{footnotesize}.
From this we obtain the important formula
\begin{equation}\label{imptform}
\Psi(z,q)=\sum_{n=0}^{\infty}z^{n}\widetilde \xi_{n}(z,q)=\sum_{\mu=0}^{\infty}\sum_{\nu=0}^{\mu}\frac{\xi^{(\mu-\nu)}_{\nu}(q;p_1,...)}{(\mu-\nu)!(m/2+\nu)_{\mu-\nu}}\left(-\frac{p_1}{2}\right)^{\mu-\nu}z^{2\nu-\mu}
\end{equation}
Now we take $z=1$ and use the identity in (\ref{imptform}) with the $p_i$'s replaced by Pontryagin classes for the vertical bundle $V\rightarrow Z$.
The assumption $p_1(Z)=0$ gives $p_1(V)=-\pi^*p_1(X)$ and since each $a_{2j}(q;p_1,...)$ is degree $4j$ in cohomology we have
$$\int_Y a_{2j}(q;p_1,...)e^{G_2(q)p_1(V)}=e^{-G_2(q)p_1(X)}\int_Y a_{2j}(q;p_1,...)=0$$
for $2j \leq (m+r)/2 -2 $.  Therefore,
$$\int_Y\psi(1,q)e^{G_2(q)p_1(V)}=\int_Y\Psi(1,q)$$
We have thus obtained
\begin{align}\label{xis}
\sch(\F ; q)&=\int_Y \prod_{i=1}^{m/2}\frac{y_i}{\sigma(y_i,q)}e^{G_2(q)y_i^2}=\int_Y\Psi(1,q)\notag\\
&=\int_Y\widetilde {\xi_0}(1,q)+\widetilde {\xi_1}(1,q)+\widetilde {\xi_2}(1,q)+...\in H^*(X,\MM^*)
\end{align}
Note that \begin{footnotesize}$\int_Y$\end{footnotesize}$\xi_{j}(q;p_1,...)\in H^{2j}(X,\M^{\frac{m}{2}+j})$.  Now, we will proceed by induction
to show that for each $j_0$
\begin{equation}\label{induct}
\int_Y \xi_{j_0}(q;p_1,...)=\sum_{j=0}^{j_0}\sum_{i=0}^{s_{j}-1}\ch_{j}(\indx \dd^{V_i})\left(\frac{p_1(X)}{2}\right)^{\lfloor(j_0-j)/2\rfloor}f^{j_0-j}_{j,i}(q)
\end{equation}
where for each $j\leq j_0$ the collection of $\phi_{j,i}(q):=f_{j,i}^{0}(q)\in\M^{\frac{m}{2}+j}$, for $i=1,...,s_j-1$, are given by Proposition \ref{genprop} and $f^{j_0-j}_{j,i}(q)\in\M^{\frac{m}{2}+j_0}$
are defined so that when $j_0-j=2\ell$ they satisfy
$$f^{2\ell}_{j,i}(q)+\sum_{n=0}^{\ell-1}\frac{D^{\ell -n}f^{2n}_{j,i}(q)}{(\ell-n)!(m/2+j+2n)_{\ell -n}}\in q^{s_{j_0}}\Q[[q]].$$
When $j_0-j$ is odd the $f^{j_0-j}_{j,i}$'s might as well be taken to be zero.  The reason for this is that $\sch_j(\F,q)$, $\ch_j(\indx \dd^{V_i})$, and the $\xi_j$ all vanish
for all odd $j$ or all even $j$ depending on whether $r=0$ or $r=2$, respectively.  It is for this reason that the appearance of the floor function
in (\ref{induct}) and the following is not all that significant.  To see why (\ref{induct}) will imply the theorem, we combine it with (\ref{imptform}) and (\ref{xis}) to see
\begin{align*}
\sch_{\leq j_0}(q)&=\left(\int_Y\sum_{j=0}^{\infty}\widetilde \xi_{j}(1,q)\right)_{\leq 2j_0}=\sum_{\mu=0}^{j_0}\sum_{\nu=0}^{\mu}\frac{1}{(\mu-\nu)!(m/2+\nu)_{\mu-\nu}}\left(\frac{p_1(X)}{2}\right)^{\mu-\nu}\int_Y\xi^{(\mu-\nu)}_{\nu}(q,p_1,...)\\
&=\sum_{\mu=0}^{ j_0}\sum_{\nu=0}^{\mu}\sum_{j=0}^{\nu}\sum_{i=0}^{s_{j}-1}\ch_{j}(\indx \dd^{V_i})\left(\frac{p_1(X)}{2}\right)^{\mu-\nu+\lfloor(\nu-j)/2\rfloor}\frac{D^{\mu-\nu}f_{j,i}^{\nu-j}(q)}{(\mu-\nu)!(m/2+\nu)_{\mu-\nu}}
\end{align*}
After setting $2\ell=\nu-j$ and $\beta = \mu-\nu$ this becomes
\begin{align*}
&\sum_{j=0}^{j_0}\sum_{i=0}^{s_{j}-1}\ch_{j}(\indx \dd^{V_i})\sum_{\ell=0}^{\lfloor(j_0-j)/2\rfloor}\left(\frac{p_1(X)}{2}\right)^{\ell}\sum_{\beta=0}^{\lfloor(j_0-j)/2\rfloor-\ell}\frac{D^{\beta}f_{j,i}^{2\ell}(q)}{\beta!(m/2+j+2\ell)_{\beta}}\left(\frac{p_1(X)}{2}\right)^{\beta}\\
&=\sum_{j=0}^{j_0}\sum_{i=0}^{s_{j}-1}\ch_{j}(\indx \dd^{V_i})\sum_{\ell=0}^{\lfloor(j_0-j)/2\rfloor}\left(\widetilde f_{j,i}^{\ell}\left(\frac{p_1(X)}{2},q\right)\right)_{\leq 2(j_0-j-2\ell)}\left(\frac{p_1(X)}{2}\right)^{\ell}\\
&=\sum_{j=0}^{j_0}\ch_{j}(\indx \dd^{V_0})\phi^{\natural}_{j,0}\left(\frac{p_1(X)}{2},q\right)_{\leq 2(j_0-j)}+...+\ch_{j}(\indx \dd^{V_{s_j-1}})\phi^{\natural}_{j,s_j-1}\left(\frac{p_1(X)}{2},q\right)_{\leq 2(j_0-j)}
\end{align*}

For $j=0$, (\ref{induct}) is an equation in $H^0(X,\M^{\frac{m}{2}})$.  As in Proposition \ref{genprop}, one can solve
$\textstyle \sch_0(\F;q)=\int_Y \xi_{0}(q;p_1,...)=\sum_{i=0}^{s_{0}-1}\ch_{0}(\indx \dd^{V_i})\phi_{0,i}(q)$.  This is just the computation of the index
of the Dirac-Ramond operator on $Y$ in terms of the index of the Dirac operator and the indices of the first $s_0-1$ twisted Dirac operators.

Now suppose that the formula (\ref{induct}) holds for $\nu < j_0$.  From (\ref{imptform}) and (\ref{xis}) we have
\begin{align}\label{indstep}
\sch_{j_0}(\F, q)&=\int_Y\sum_{\nu=0}^{j_0}\frac{\xi_{j}^{(j_0-\nu)}(q;p_1,...)}{(j_0-\nu)!(m/2+\nu)_{j_0-\nu}}\left(-\frac{p_1}{2}\right)^{j_0-\nu}\\
&=\sum_{\nu=0}^{j_0}\frac{1}{(j_0-\nu)!(m/2+\nu)_{j_0-\nu}}\left(\frac{p_1(X)}{2}\right)^{j_0-\nu}\int_Y\xi^{(j_0-\nu)}_{\nu}(q;p_1,...).\notag
\end{align}
Applying the induction hypothesis (\ref{induct}) to the terms with $\nu< j_0$ gives
\begin{align}\label{gettingsomewhere}
\Upsilon:&=\sch_{j_0}(\F ;q)-\int_Y\xi_{j_0}(q;p_1,...) \notag \\
&=\sum_{\nu=0}^{j_0-1}\frac{1}{(j_0-\nu)!(m/2+\nu)_{j_0-\nu}}\sum_{j=0}^{\nu}\sum_{i=0}^{s_{j}-1}\ch_{j}(\indx \dd ^{V_i})D^{j_0-\nu}f_{j,i}^{\nu-j}(q)\left(\frac{p_1(X)}{2}\right)^{\lfloor(j_0-j)/2\rfloor}\notag\\
&=\sum_{j=0}^{j_0-1}\sum_{i=0}^{s_{j}-1}\ch_{j}(\indx \dd ^{V_i})\sum_{\ell=0}^{\lfloor(j_0-1-j)/2\rfloor}\frac{D^{j_0-j-2\ell}f_{j,i}^{2\ell}(q)}{(j_0-j-2\ell)!(m/2+j+2\ell)_{j_0-j-2\ell}}\left(\frac{p_1(X)}{2}\right)^{\lfloor(j_0-j)/2\rfloor}\notag
\end{align}
where in the third line we set $2\ell = \nu-j$.

From the definition we have
$$\Sch_{j_0}(\F ;q)=\sum_{\nu=0}^{\infty}q^{\nu}\ch_{j_0}(\indx \dd ^{V_\nu}).$$
Multiplying this equation by $\frac{\eta (q)^m}{q^{m/24}}$ and combining it with (\ref{gettingsomewhere}) gives
\begin{align*}
\int_Y \xi_{j_0}(q;p_1,...)=\frac{\eta (q)^m}{q^{m/24}}\sum_{\nu=0}^{\infty}q^{\nu}\ch_{j_0}(\indx \dd ^{V_{\nu}})-\Upsilon\in H^{2j_0}(X,\M^{\frac{m}{2}+j_0}).
\end{align*}
We proceed as in Proposition \ref{genprop} and find a basis for $\M^{\frac{m}{2}+j_0}$
of the form $\{  \phi_{j_0,0}(q),..., \phi_{j_0,s_{j_0}-1}(q)\}$ which satisfies $q^{m/24}\frac{\phi_{j_0,i}(q)}{\eta(q)^m}=q^i+\mathcal O (q^{s_{j_0}})$.
Then $$\frac{\eta (q)^m}{q^{m/24}}\sum_{i=0}^{\infty}q^{i}\ch_{j_0}(\indx \dd ^{V_i})=\ch_{j_0} (\indx \dd ^{V_0})\phi_{j_0,0}(q)+...+\ch_{j_0} (\indx \dd ^{V_{s_{j_0}-1}})\phi_{j_0,s_{j_0}-1}(q)\ \ (\bmod \ q^{s_{j_0}}).$$
Let $\Omega \in H^{2j_0}(X,\M^{\frac{m}{2}+j_0})$ denote the RHS of the previous equation.
Then
\begin{align*}
\left(\frac{\eta (q)^m}{q^{m/24}}\sum_{i=0}^{\infty}q^{i}\ch_{j_0}(\indx \dd ^{V_i})-\Omega\right)-\Upsilon
\end{align*}
is equal to $-\Upsilon$ up to order $q^{s_{j_0}-1}$ as an element of $H^*(X,\Q)[[q]]$.
For each $(j,i)$ we can find an $f^{j_0-j}_{j,i}(q)\in\M^{\frac{m}{2}+j_0}$ such that
$$f^{j_0-j}_{j,i}(q)=-\sum_{\ell=0}^{\lfloor (j_0-1-j)/2\rfloor}\frac{D^{j_0-j-2\ell}f_{j,i}^{2\ell}(q)}{(j_0-j-2\ell)!(m/2+j+2\ell)_{j_0-j-2\ell}} (\bmod \ q^{s_{j_0}})$$
which is uniquely defined since $\M^{\frac{m}{2}+j_0}$ is $s_{j_0}$ dimensional.  Thus, we have
\begin{align*}
\Lambda &:=\sum_{j=0}^{j_0-1}\sum_{i=0}^{s_{j}-1}\ch_{j}(\indx \dd ^{V_i})\left(\frac{p_1(X)}{2}\right)^{j_0-j}f^{j_0-j}_{j,i}(q)\\
&=-\sum_{j=0}^{j_0-1}\sum_{i=0}^{s_{j}-1}\ch_{j}(\indx \dd ^{V_i})\sum_{\ell=0}^{\lfloor(j_0-1-j)/2\rfloor}\frac{D^{j_0-j-2\ell}f_{j,i}^{2\ell}(q)}{(j_0-j-2\ell)!(m/2+j+2\ell)_{j_0-j-2\ell}}\left(\frac{p_1(X)}{2}\right)^{\lfloor(j_0-j)/2\rfloor} (\bmod \ q^{s_{j_0}})
\end{align*}
Now, since $\int_Y \xi_{j_0} = \Omega - \Lambda \ (\bmod \ q^{s_{j_0}})$ and each quantity is in $H^{2j_0}(X,\M^{\frac{m}{2}+j_0})$ we have equality for all orders of $q$, i.e.
\begin{align}
&\int_Y \xi _{j_0} (q,p_1,...)=\sum_{i=0}^{s_{j_0}}\ch_{j_0} (\indx \dd ^{V_i})\phi_{j_0,i}(q)+\sum_{j=0}^{j_0-1}\sum_{i=0}^{s_{j}-1}\ch_{j}(\indx \dd ^{V_i})\left(\frac{p_1(X)}{2}\right)^{\lfloor(j_0-j)/2\rfloor}f^{j_0-j}_{j,i}(q)
\end{align}
and this is the induction step we wanted to show in (\ref{induct}).

\end{proof}

\section{The $E_8$ Bundle}
Our formal version of the Dirac-Ramond operator should arise as the restriction of an actual operator on loop space.  And it is believed (see \cite{Brylinski}, for instance)
that the actual operator on loop space should fit into the framework of a $\Diff (S^1)$ equivariant $K$-theory of loop space.  It is desirable then for the index bundle $\textstyle \indx\DD := \sum_{n=0}^{\infty}q^n \indx \dd ^{V_n}$
to be the restriction of a ``Virasoro equivariant" vector bundle on loop space.
A nice class of algebras whose representations also furnish representations for the
Virasoro algebra are affine Lie algebras.  Below we will show that under some stringent conditions we can identify the index bundle with a bundle associated to a representation
of affine $E_8$.

\subsection{Principal $E_8$ Bundles}\label{e8bdlsect}
The homotopy groups for $E_8$ are known to satisfy $\pi_i(E_8)= 0$ for $1 \leq i \leq 14$ except for $\pi_3 (E_8)\simeq \Z$.
So for $i\leq 15$ the only nonzero homotopy group of $BE_8$ is $\pi_4(BE_8)\simeq \Z$.  This makes $BE_8$ an ``approximate" $K(\Z, 4)$.  That is, for
manifolds $X$ of dimension at most $14$ the cellular approximation theorem gives an isomorphism

\begin{equation}\label{corr}
[X,BE_8]\simeq [X,K(\Z,4)]\simeq H^4(X,\Z).
\end{equation}
This has the effect that, in low dimensions, principal $E_8$ bundles over $X$ are in bijective correspondence with the elements of its fourth cohomology.
Let $P\rightarrow X$ be a principal $E_8$ bundle over $X$.  If $EE_8\rightarrow BE_8$ denotes the universal principal $E_8$ bundle over the classifying space
$BE_8$, then $P=\gamma^*EE_8$ for some $\gamma :X\rightarrow BE_8$.  The bijective correspondence in (\ref{corr}) associates the principal bundle $P$ to the cohomology class
$\omega_{\gamma} = \gamma^*(u)$ where $u$ is the generator of
$H^4(BE_8,\Z)$.

The adjoint representation of $E_8$ is a $248$ dimensional unitary representation.  Let
$\rho : E_8 \rightarrow U(248)$ denote this representation.  In fact, if we compose $\rho$ with the determinant map then we get a map from $E_8$ into $U(1)$.  Since $E_8$ is simple, the kernel of this map must be all of $E_8$.  Thus we actually
have image$(\rho)\subset SU(248)$ and we see that $P$ is also a principal $SU(248)$ bundle.  The goal is now to compute the Chern classes of this bundle.

To obtain the Chern classes we need a map $X\rightarrow BSU(248)$.  The representation $\rho$ induces a map $B\rho : BE_8 \rightarrow BSU(248)$.  The map we need then is
given by the composition $B\rho \circ \gamma : X \rightarrow BE_8 \rightarrow BSU(248)$.  Since $H^2(BSU(248),\Z)=0$ we trivially have $c_1(P)=0$.  Now,
$$c_2(P)= (B\rho \circ \gamma)^*(c_2)=\gamma^*B\rho^*c_2.$$
Since $H^4(BSU(248),\Z)$ and $H^4(BE_8,\Z)$ are both canonically isomorphic to $\Z$, any homomorphism between them is determined by a single integer.  The integer
induced by the adjoint representation is known as the Dynkin index of $E_8$ and has been computed to be $60$ (see \cite{Totaro} and references therein).  We restate all this in the following proposition.
\begin{prop}
Let $\rho:E_8 \rightarrow SU(248)$ be the adjoint representation and $c_2$ and $u$ be the generators of $H^4(BSU(248),\Z)$ and $H^4(BE_8,\Z)$, respectively.  Then
\begin{align}
B\rho^*:H^4(BSU(248),\Z) &\rightarrow H^4(BE_8,\Z)\notag \\
c_2 &\mapsto 60u
\end{align}
\end{prop}
\noindent We now see that
$$c_2(P)= 60\omega^*(u).$$
It is well known that when $X$ is spin, $p_1(X)$ is even.  Since we then have \begin{footnotesize}$\frac{p_1(X)}{2}$\end{footnotesize}$\in H^4(X,\Z)$,
we can choose $\omega_{\gamma} = -$\begin{footnotesize}$\frac{p_1(X)}{2}$\end{footnotesize} and therefore $c_2(P)=-30p_1(X)$.  Let $W=P\times_{\rho}\C^{248}$ be the complex vector bundle over $X$ associated to the representation $\rho$.
The Chern character of $W$ is the same as the Chern character for $P$, since $W$ can also be viewed as being associated to $P$ as a $SU(248)$ bundle using the standard representation.  Working with
formal Chern variables $x_1,...,x_{248}$ we see
$$x_1^2+...+x_{248}^2=(x_1+...+x_{248})^2-2\sum_{i<j}x_ix_j=c_1(W)^2-2c_2(W)=-2c_2(W)$$
and thus
$$\ch_2(W)=\frac{1}{2}(x_1^2+...+x_{248}^2)=-c_2(W)=30p_1(X).$$

It is a result of Atiyah-Hirzebruch \cite{ATHIRZ} that for any compact simple group $G$, the (completed) representation ring $R(G)$ is isomorphic to
$K(BG)$.  Invoking the Chern character gives the isomorphism $R(E_8)\otimes \Q \simeq H^*(BE_8,\Q)$.  Under this identification a representation $\Lambda$
is identified with $\ch (EE_8\times _{\Lambda}\C^r)$ where $r$ is the dimension of the representation $\Lambda$.
From some character calculations in \cite{Totaro} we can see then
\begin{align}\label{e8chars}
\ch (EE_8\times_{\rho}\C^{248})&=248+60u+6u^2+...
\end{align}
 Using $\gamma$ to pull back to $X$ we obtain,
\begin{equation}\label{e8w8}
\ch (W) = 248 +30p_1(X)+\frac{3}{2}p_1(X)^2+...
\end{equation}

\subsection{The Basic Representation of Affine $E_8$}
Let $\g$ be a complex finite-dimensional simple Lie algebra and $\langle \cdot,\cdot \rangle$ be the Killing form on $\g$.  We will also take $\g$ to be simply laced, i.e. of type $A_n$,$D_n$, or $E_n$.
The affine Lie algebra $\widehat{\g}$ corresponding to $\g$ is
$$\widehat{\g}=\C[t,t^{-1}]\otimes \g \oplus \C K \oplus \C d$$
where $d=t\frac{d}{dt}$ and the bracket is defined by
\begin{align*}
&[f(t)\otimes X+\alpha K + \mu d,g(t)\otimes Y+\beta K + \nu d]\\
=&f(t)g(t)\otimes [X,Y]+\langle X,Y\rangle \text{Res }_{t=0}\left(\frac{df}{dt}(t)g(t)\right)K + \mu (dg)(t)\otimes Y - \nu (df)(t)\otimes X.
\end{align*}
Via the identification $t=e^{i\theta}$ the first summand appearing in $\widehat{\g}$ can be thought of as loops on $\g$ with finite Fourier expansion.  Writing $G$ for the compact
simply connected Lie group corresponding to $\g$, these loops
exponentiate to polynomial loops on $G$.  The second summand provides a central extension of this algebra.
In terms of $\theta$, the operator $d$ is $-i$\begin{footnotesize}$\frac{d}{d\theta}$\end{footnotesize}.  This operator exponentiates to rigid rotations on the circle and so the third
summand provides a semidirect product with the infinitesimal generator of such transformations.
Of course, $\g$ is a subalgebra via the identification of $\g$ with $1\otimes \g$.

Of all the irreducible representations of $\widehat{\g}$ there is a nontrivial one which is simplest in some ways.
This representation $V(\Lambda_0)$ is known as the basic representation and contains a highest weight vector $v_0$ satisfying
$$Kv_0=v_0 \text { and } \left(\C[t]\otimes \g  \oplus \C  d\right)v_0=0.$$
Let $W_{n}=\{v\in V(\Lambda_0) | dv = -nv\}$.  Since $[\g, d]=0$ each $W_n$ is a representation for $\g$.
The character of the basic representation $V(\Lambda_0)$ is given by (see for instance \cite{KacBook}, or more in the case at hand \cite[equation (2)]{Kac})
\begin{equation}\label{charbas}
\text{char } (V(\Lambda_0))(q,z_1,...,z_r)=q^{r/24}\frac{\Theta_{\mathfrak g}(q,z_1,...,z_r)}{\eta(q)^r}
\end{equation}
where $r$ is the rank of $\mathfrak g$, $(z_1,...,z_r)$ represents a point (after choosing a basis) in the Cartan subalgebra, and the theta function is defined on the root lattice $Q$ by

$$\Theta_{\mathfrak g}(q,z_1,...,z_r)=\sum_{\gamma \in Q}e^{2\pi i \langle \gamma, \vec{z}\rangle}q^{||\gamma ||^2/2}.$$
We now specify all of this to the case where $\mathfrak g = E_8$.
The representation $V(\Lambda_0)$
breaks up in terms of the $W_n$'s as a sequence of finite dimensional representations for $E_8$ (the algebra or the group) as
\begin{equation}\label{lvl1}
V(\Lambda_0)=1+W_1 q + W_2 q^2+...
\end{equation}
where $1$ denotes the trivial one-dimensional representation.  It is a fact that $W_1$ is the adjoint representation $\rho$ for $E_8$.
Recall from the previous section there is an $E_8$ bundle $P$ over $X$ corresponding to the cohomology class $-$\begin{footnotesize}$\frac{p_1(X)}{2}$\end{footnotesize}.
For each $n>0$ define the associated vector bundles $\underline{W}_n=P\times_{\rho_n}W_n$ over $X$ and put
\begin{equation}\label{Vdefn}
\mathcal{V}=1_{\C}+\underline{W}_1 q + \underline{W}_2 q^2+...\in K(X)[[q]]
\end{equation}
where $1_{\C}\rightarrow X$ is the trivial one-dimensional complex vector bundle.

Using (\ref{e8w8}) we see that
\begin{equation}\label{e8w82}
\ch(\mathcal V)=1+(248 +30p_1(X))q+...
\end{equation}
and the rest of the terms are at least degree $8$ in cohomology or at least degree $2$ in $q$.

In \cite{Gannon} it is shown that there is a basis for the $E_8$ root lattice such that
\begin{equation}
\Theta_{E_8}(q,z_1,...,z_8)=\frac{1}{2}\left(\prod_{i=1}^{8}\theta_2(2\pi iz_i,q)+\prod_{i=1}^{8}\theta_3(2\pi iz_i,q)+\prod_{i=1}^{8}\theta_4(2\pi iz_i,q)+\prod_{i=1}^{8}\theta(2\pi i z_i,q)\right)
\end{equation}
where $\theta_2(z,q),\theta_3(z,q),\theta_4(z,q)$ are the other three classical (even) Jacobi theta functions whose definitions are given in the appendix.
Now, let $$H(\tau,z_1,...,z_8)=\frac{1}{2}e^{G_2(\tau)(z_1^2+...+z_8^2)}\left(\prod_{i=1}^{8}\theta_2(z_i,\tau)+\prod_{i=1}^{8}\theta_3(z_i,\tau)+\prod_{i=1}^{8}\theta_4(z_i,\tau)\right).$$
Notice that since $\theta(0,\tau)=0$, it follows that $\theta(z,\tau) = \mathcal O (z)$ and hence
$$\Theta_{E_8}(\tau,z_1,...,z_8)=e^{-G_2(\tau)((2\pi i z_1)^2+...+(2\pi i z_8)^2)}H(\tau, 2\pi i z_1,...,2\pi i z_8) + \mathcal O (z^8).$$
Using the classical transformation formulas for the Jacobi theta functions and for $G_2$ one sees that
\begin{equation}\label{thetatrans}
H\left(\frac{a\tau+b}{c\tau+d},\frac{z_1}{c\tau+d},...,\frac{z_8}{c\tau+d}\right)=(c\tau+d)^4H(\tau,z_1,...,z_8).
\end{equation}
Let $e_i$ be the $i$th elementary symmetric polynomial in $z_1^2,...,z_8^2$.
Since each theta function in $H(\tau,z_1,...,z_8)$ is even, it can be expanded in terms of the elementary symmetric functions $e_i$'s
$$H(\tau,z_1,...,z_8)=a_0(\tau)+a_{1,1}(\tau)e_1+a_{2,1}(\tau)e_1^2+a_{2,2}(\tau)e_2+...$$
It follows from (\ref{thetatrans}) that each $a_{i,j}$ is a modular form of weight $4+2i$.  For $i<4$, the space of modular forms of weight $4+2i$ is one dimensional.  So
to determine $a_{i,j}$ one need only calculate its constant term.  This can be done very easily using Mathematica.  One finds
\begin{equation}
H(\tau,z_1,...,z_8)=E_4(\tau)-\frac{1}{12}E_6(\tau)\frac{1}{2}e_1+\frac{1}{2!\cdot12^2}E_8(\tau)(\frac{1}{2}e_1)^2-\frac{1}{3!\cdot12^3}E_{10}(\tau)(\frac{1}{2}e_1)^3+\mathcal O (z^8).
\end{equation}
Then
\begin{align}\label{thetachars}
&\frac{1}{2}\left(\prod_{i=1}^{8}\theta_2(z_i,\tau)+\prod_{i=1}^{8}\theta_3(z_i,\tau)+\prod_{i=1}^{8}\theta_4(z_i,\tau)\right) \notag\\
&=e^{-G_2(\tau)e_1}\left(E_4(\tau)-\frac{1}{12}E_6(\tau)\frac{1}{2}e_1+\frac{1}{2!\cdot12^2}E_8(\tau)(\frac{1}{2}e_1)^2-\frac{1}{3!\cdot12^3}E_{10}(\tau)(\frac{1}{2}e_1)^3\right)+\mathcal O (z^8)\notag\\
&=E_4(\tau)+\frac{1}{4}E_4'(\tau)\left(\frac{e_1}{2}\right)+\frac{1}{2!\cdot 4\cdot 5}E_4''(\tau)\left(\frac{e_1}{2}\right)^2+\frac{1}{3!\cdot 4\cdot 5\cdot 6}E_4'''(\tau)\left(\frac{e_1}{2}\right)^3+\mathcal O (z^8)
\end{align}
Using (\ref{charbas}) together with (\ref{thetachars}) we get
\begin{equation}\label{b4pullback}
\ch (V(\Lambda_0)) = \frac{q^{1/3}}{\eta(q)^8}\left(E_4(\tau)+\frac{1}{4}E_4'(\tau)\left(\frac{e_1}{2}\right)+\frac{1}{2!\cdot 4\cdot 5}E_4''(\tau)\left(\frac{e_1}{2}\right)^2+\frac{1}{3!\cdot 4\cdot 5\cdot 6}E_4'''(\tau)\left(\frac{e_1}{2}\right)^3\right)+\mathcal O (z^8)
\end{equation}
With the identification $R(E_8)\otimes \Q \simeq H^*(BE_8,\Q)$, $\ch(V(\Lambda_0))$ is an element of $H^*(BE_8,\Q)[[q]]$.  Proceeding as in
the proof of Theorem \ref{thm31} one could see further that \begin{footnotesize}$\frac{\eta(q)^8}{q^{1/3}}$\end{footnotesize}$\ch (V(\Lambda_0))\in H^*(BE_8,\MM^*)$.
Now we use the map $\gamma :X\rightarrow BE_8$ corresponding to $-$\begin{footnotesize}$\frac{p_1(X)}{2}$\end{footnotesize} (see Section \ref{e8bdlsect}) to pull
back (\ref{b4pullback}) to $H^*(X,\Q)[[q]]$ and compare the degree $4$ element of cohomology appearing as the coefficient of $q$ with that in (\ref{e8w8}).  We see that $\gamma^*e_1=p_1(X)$.  Thus
\begin{align}\label{charRHS}
\ch(\mathcal{V})=\frac{q^{1/3}}{\eta(q)^8} \{&E_4(q)+\frac{1}{4}E_4'(q)\left(\frac{p_1(X)}{2}\right)\\
&+\frac{1}{2!\cdot 4\cdot 5}E_4''(q)\left(\frac{p_1(X)}{2}\right)^2+\frac{1}{3!\cdot 4\cdot 5\cdot 6}E_4'''(q)\left(\frac{p_1(X)}{2}\right)^3+...\in H^*(X,\Q)[[q]]\notag
\end{align}
An application of Theorem \ref{bigthm}, or more directly (\ref{start}), now gives the following result.
\begin{thm}
Let $\F=(\pi, Z, X)$ be a string family of compact spin manifolds having fibers $Y_x=\pi^{-1}(x)$ of dimension $8$.  Suppose also that $X$ is a compact spin manifold of dimension less
than $16$.  If $\ch_2(\indx \dd)=\ch_4(\indx \dd)=\ch_{6}(\indx \dd)=0$
then the Chern character of the index bundle for the family of Dirac-Ramond operators satisfies
\begin{equation}
\Sch(\F ; q)=\ch (\C^{\nu_0}\otimes \mathcal{V})
\end{equation}
where $\nu_0$ is the index of the Dirac operator on $Y$ and $\mathcal{V}\in K(X)[[q]]$ is constructed as above.
\end{thm}

\appendix

\section{Modular Forms and some Related Functions}
In this appendix, we point out some of our conventions and recall some standard facts.  Most all our definitions of functions are consistent with \cite{HBJ}.  As always,
$q=e^{2\pi i \tau}$.

The (normalized) Eisenstein series of weight $2k$ is given by
\begin{equation}
E_{2k}(q)=1-\frac{4k}{B_{2k}}\sum_{n=1}^{\infty}(\sum_{d|n} d^{2k-1})q^n
\end{equation}
where $B_{2k}$ is the $2k$th Bernoulli number.  For instance,
\begin{align*}
E_2(q)&=1-24q-72q^2-96q^3-...\\
E_4(q)&=1+240q+2160q^2+6720q^3+...\\
E_6(q)&=1-504q-16632q^2-122976q^3-...
\end{align*}
There is also the notable weight $12$ modular form $$\Delta (q) =\frac{E_4(q)^3-E_6(q)^2}{1728}=q-24q^2+252q^3+...$$
The differential operator $D=q\frac{d}{dq}=\frac{1}{2\pi i}\frac{d}{d\tau}$ maps (quasi)modular forms of weight $k$ to quasimodular forms of weight $k+2$.  The following
formulas are useful
\begin{align*}
DE_2(q)&=\frac{E_2(q)^2-E_4(q)}{12}\\
DE_4(q)&=\frac{E_2(q)E_4(q)-E_6(q)}{3}\\
DE_6(q)&=\frac{E_2(q)E_6(q)-E_4(q)^2}{2}\\
D\Delta (q)&=\Delta(q) E_2(q).
\end{align*}
The Dedekind eta function is defined by
$$\eta(q)=q^{1/24}\prod_{n=0}^{\infty}(1-q^n).$$
and satisfies the identity
$$\eta(q)^3=\theta ' (0,q)$$
where $\theta (z,q)$ is the Jacobi theta function
$$\theta (z,q)=2q^{1/8}\sinh (z/2) \prod_{n=1}^{\infty}(1-q^ne^z)(1-q^ne^{-z})(1-q^n).$$
The other three Jacobi theta functions are
\begin{align*}
\theta_2 (z,q)&=2q^{1/8}\cosh (z/2) \prod_{n=1}^{\infty}(1+q^ne^z)(1+q^ne^{-z})(1-q^n)\\
\theta_3 (z,q)&=\prod_{n=1}^{\infty}(1+q^{n-\frac{1}{2}}e^z)(1+q^{n-\frac{1}{2}}e^{-z})(1-q^n)\\
\theta_4 (z,q)&=\prod_{n=1}^{\infty}(1-q^{n-\frac{1}{2}}e^z)(1-q^{n-\frac{1}{2}}e^{-z})(1-q^n).
\end{align*}
Another important elliptic function is the Weierstrass sigma function, which is typically defined by
$$\sigma (z,\tau)=\prod_{0\neq\gamma \in 2\pi i (\Z + \tau \Z)}\left(1-\frac{x}{\gamma}\right)e^{\frac{x}{\gamma}+\frac{1}{2}(\frac{x}{\gamma})^2}.$$
For our purposes, this is not a useful expression.  We make more use of it in the following identities
\begin{align}\label{sigma}
\sigma(z,\tau)&=\frac{\theta(z,\tau)}{\theta ' (0,\tau)}e^{G_2(\tau)z^2}\notag\\
&=z\exp\left(-\sum_{n=2}^{\infty}\frac{2}{2n!}G_{2n}(\tau)z^{2n}\right )
\end{align}
where $G_{2n}(\tau)$ is the (unnormalized) Eisenstein series
$$G_{2k}(\tau)=-\frac{B_{2k}}{4k}E_{2k}(\tau).$$
For instance,$G_2(\tau)=-\frac{1}{24}E_2(\tau)$, $G_4(\tau)=\frac{1}{240}E_4(\tau)$, and $G_6(\tau)=-\frac{1}{504}E_6(\tau)$.
The Eisenstein series $E_2$ is not a modular form.  It is quasimodular satisfying
\begin{equation}\label{E2trans}
E_2\left(\frac{a\tau+b}{c\tau + d}\right)=(c\tau + d)^2E_2(\tau)+\frac{12c}{2\pi i}(c \tau + d)
\end{equation}
and so
\begin{equation}\label{G2trans}
G_2\left(\frac{a\tau+b}{c\tau + d}\right)=(c\tau + d)^2G_2(\tau)-\frac{c}{4\pi i}(c \tau + d).
\end{equation}
The definition of a Jacobi-like form was given in Section \ref{pfsect}.  We want to point out why the coefficients of a Jacobi-like form are modular
when the index $\lambda$ is equal to $0$ and quasimodular when $\lambda \neq 0$.  Let $F(z,\tau)=\sum_{n=0}^{\infty}\chi_{2n}(\tau)z^{2n} \in \mathcal J^+_{k,\lambda}$.
Set $H(z,\tau)=e^{2G_2(\tau)\lambda z^2}F(z,\tau)$.  Then
\begin{align*}
H\left(\frac{z}{c\tau + d},\frac{a\tau+b}{c\tau + d}\right)&=\exp\left(\left((c\tau + d)^2G_2(\tau)-\frac{c}{4\pi i}(c \tau + d)\right)2\lambda(\frac{z}{c\tau + d})^2\right)(c\tau +d)^k \exp\left(\frac{c\lambda}{c\tau +d}\frac{z^2}{2\pi i}\right)F(z,\tau)\\
&=(c\tau +d)^k H(z,\tau).
\end{align*}
Now expand $H(z,\tau)$ as a power series $\sum_{n=0}^{\infty} \widetilde \chi_{2n} (\tau) z^{2n}$.  The previous equation becomes
$$\sum_{n=0}^{\infty} \widetilde \chi_{2n} \left(\frac{a\tau+b}{c\tau + d}\right) \left(\frac{z}{c\tau + d}\right)^{2n}=(c\tau +d)^k\sum_{n=0}^{\infty} \widetilde \chi_{2n} (\tau) z^{2n}$$
so that each $\widetilde {\chi}_{2n}$ is a modular form of weight $k+2n$.  Then $F(z,\tau)=e^{-2G_2(\tau)\lambda z^2}\sum_{n=0}^{\infty} \widetilde \chi_{2n} (\tau) z^{2n}$, and
it follows from this that each $\chi_{2n}$ is a quasimodular form of weight $k+2n$.

Given $F(z,\tau)=\sum_{n=0}^{\infty}\chi_{2n}(\tau)z^{2n}\in \mathcal J_{0,\lambda}^+$ then $\chi_0(\tau)$ is a modular form of weight $0$ and hence
constant.  Assume $\chi_0(\tau)=1$.  Then $\prod_{i=1}^{m/2}F(y_i,\tau)$ is expressible in terms of the elementary symmetric functions $p_1,...,p_{m/2}$
in the variables $y_1^2,...,y_{m/2}^2$.  It follows then that if we write
\begin{equation}\label{expand}
\prod_{i=1}^{m/2}F(y_i,\tau)=1+a_{1,1}(\tau)p_1+a_{2,1}(\tau)p_1^2+a_{2,2}(\tau)p_2+...
\end{equation}
then each $a_{i,j}(\tau)$ is a modular form of weight $2i$ if $\lambda=0$ and a quasimodular form of the same weight otherwise.

\section{The Computation}
In this section we show how the computations in section \ref{compmotiv} were done and hope to elucidate the proof in section \ref{pfsect}.  We
assume the setup from the previous sections.  Namely, we have a string family $Z\rightarrow X$ parameterizing the compact spin manifolds $Y_x=\pi^{-1}(x)$  and
$V\rightarrow Z$ is the vertical bundle.  The string condition on $Z$ allows us to make much use of (\ref{p1n}).
For simplicity, we will restrict to when the dimension of $Y$ is $8$.  All, computations below were done with the help of
Mathmematica.

We recall the formula for the $\widehat A$-class.  For a vector bundle $V$ with Pontryagin classes $p_1(V),...,p_{m/2}(V)$ we have
\begin{align}\label{ahat}
\hat A (V)&=1-\frac{1}{24}p_1(V)+\frac{7p_1(V)^2-4p_2(V)}{5760}+\frac{-31p_1(V)^3+44p_1(V)p_2(V)-16p_3(V)}{967680}\\
&+\frac{381p_1(V)^4-904p_1(V)^2p_2(V)+208p_2(V)^2+512p_1(V)p_3(V)-192p_4(V)}{464486400}+... \notag
\end{align}
We will start with the case when dim $Y=8$.  Then by the usual Atiyah-Singer index theorem for families of Dirac operators we have
\begin{align}\label{classes}
\ch_0 (\indx \dd)&=\int_Y -\frac{4p_2(V)}{5760}\\
\ch_2 (\indx \dd)&=\int_Y \frac{44p_1(V)p_2(V)-16p_3(V)}{967680}\notag\\
&\vdots \notag
\end{align}
Recall that in the notation above we have $V_1=V_{\C}$.  Some manipulations with formal Chern variables show
\begin{align}
\ch (V_{\C})&=8+p_1(V)+\frac{p_1(V)^2-2p_2(V)}{12}+\frac{p_1(V)^3-3p_1(V)p_2(V)+3p_3(V)}{360}\\
&+\frac{p_1(V)^4-4p_1(V)^2p_2(V)+2p_2(V)^2+4p_1(V)p_3(V)-4p_4(V)}{20160}+...\notag
\end{align}
After multiplying this by (\ref{ahat}), the index theorem (\ref{ordfamthm}) gives
\begin{align*}
\ch_0(\indx \dd ^{V_1})&=\int_Y -\frac{31p_2(V)}{180}\\
\ch_2(\indx \dd ^{V_1})&=\int_Y \frac{-13 p_1(V)p_2(V) + 62 p_3(V)}{7560}\\
&\vdots
\end{align*}
Notice that if $p_1(X)=0$ then $p_1(V)=-\pi^*p_1(X)=0$ and then $\ch_0(\indx \dd^{V_1})=248 \ch_0 (\indx \dd)$ and
$\ch_2(\indx \dd^{V_1})=-496 \ch_2 (\indx \dd)$ agreeing with (\ref{ch0s}) and (\ref{ch4s}), respectively.  However,
when $p_1(V)=-\pi^*p_1(X)\neq 0$ we have
\begin{align*}
\ch_2(\indx \dd^{V_{\C}})&=\int_Y-496\frac{44p_1(V)p_2(V)-16p_3(V)}{967680} +30\frac{-4p_2(V)}{5760}(-p_1(V)) \\
&=-496\int_Y\frac{44p_1(V)p_2(V)-16p_3(V)}{967680}+30p_1(X)\int_Y-\frac{4p_2(V)}{5760}\\
&=-496\ch_2 (\indx \dd)+30\ch_0 (\indx \dd)p_1(X)
\end{align*}
agreeing with (\ref{p1not0}).

The equation (\ref{start}) follows from Theorem \ref{bigthm}.  Here we will also show how it follows from direct computation.  First, recall
\begin{align*}
\widehat a(V,\tau)e^{G_2(\tau)p_1(V)}&=\prod_{i=1}^{m/2}\frac{y_i}{\sigma(y_i,\tau)}e^{G_2(\tau)y_i^2}\\
&=\prod_{i=1}^{m/2}\exp(\sum_{n=1}^{\infty}\frac{2}{2n!}G_{2n}(\tau)y_i^{2n})=\exp\left(\sum_{n=1}^{\infty}\frac{2}{2n!}G_{2n}(\tau)(y_1^{2n}+...+y_{m/2}^{2n})\right)
\end{align*}
and after making use of the Newton identities which express the power sums in the basis of elementary symmetric polynomials this can be written
\begin{align}
\widehat a (V,\tau)&e^{G_2(\tau)p_1(V)}\\
&=\exp\left(-\frac{E_2(q)}{24}p_1(V)+\frac{E_4(q)}{2880}(p_1(V)^2-2p_2(V))+\frac{E_6(q)}{181440}(p_1(V)^3-3p_1(V)p_2(V)+3p_3(V))+...\right)\notag
\end{align}
We will write $p_i$ for $p_i(V)$ in the following.  Expanding this out we get something of the form
\begin{align*}
\widehat a (V,\tau)&e^{G_2(\tau)p_1(V)}=F(p_1,q)-\frac{E_4(q)p_2}{1440}+\frac{84E_2(q)E_4(q)p_1p_2+48E_6(q)p_1p_2-48E_6(q)p_3}{2903040}+...
\end{align*}
where $F(p_1,q)$ is some expression only depending on powers of $p_1$.  Hence,
\begin{align*}
\sch_{\leq 4} (q)&=\int_Y \widehat a (V,\tau)e^{G_2(\tau)p_1(V)}=\int_Y -\frac{E_4(q)p_2}{1440}+\frac{84E_2(q)E_4(q)p_1p_2+48E_6(q)p_1p_2-48E_6(q)p_3}{2903040}\\
&=\int_Y-\frac{p_2}{1440}\left(E_4(q)+\frac{1}{4}\frac{E_2(q)E_4(q)-E_6(q)}{3}(-\frac{p_1}{2})\right)+\frac{44p_1(V)p_2(V)-16p_3(V)}{967680}E_6(q)\\
&=\ch_0 (\indx \dd)\left(E_4(q)+\frac{1}{4}DE_4(q)\left(\frac{p_1(X)}{2}\right)\right)+\ch_2 (\indx \dd)E_6(q).
\end{align*}
Continuing this to higher degrees in cohomology gives (\ref{start}).

\subsection*{Acknowledgements}
I am thankful to my PhD advisor Orlando Alvarez for suggesting
this problem and for all of his guidance along the way. I am also
thankful to Nikolai Saveliev for many helpful discussions and
corrections to earlier versions of this paper. Many thanks also go
to Anatoly Libgober for also going through the earlier versions
and for many valuable comments.

\vspace{10mm}
\bibliographystyle{plain}
\bibliography{myrefs}

\end{document}